\documentclass[psamsfonts]{amsart}
\usepackage{amsmath,graphicx,amsfonts,amssymb,amsthm,hyperref,amscd,verbatim,tikz,tikz-cd, url}
\usepackage[all,arc]{xy}
\usepackage{enumerate}
\usepackage{mathrsfs}
\providecommand{\U}[1]{\protect\rule{.1in}{.1in}}
\def\theenumi{\arabic{enumi}}

\def\theenumii{\alph{enumii}}
\def\p@enumii{\theenumi.}

\def\theenumiii{\arabic{enumiii}}
\def\p@enumiii{(\theenumi)(\theenumii)}

\def\p@enumiv{\p@enumiii.\theenumiii}
\newcommand{\dmo}{\DeclareMathOperator}
\newcommand{\Z}{\mathbb{Z}}
\newcommand{\N}{\mathbb{N}}

\newcommand{\Q}{\mathbb{Q}}

\newcommand{\Ca}{\mathcal{C}}

\newcommand{\G}{\mathcal{G}}
\newcommand{\as}{\text{*}}
\newcommand{\dt}{\bullet}
\newcommand{\mr}{\mathfrak{m}}
\newcommand{\Sn}{\mathfrak{S}}
\newcommand{\arXiv}[1]{\href{http://arxiv.org/abs/#1}{\nolinkurl{arXiv:#1}}}
\newcommand{\arXivV}[2]{\href{http://arxiv.org/abs/#1}{\nolinkurl{arXiv:#1v#2}}}
\dmo\FI{FI}
\dmo\FB{FB}
\dmo\VI{VI}
\dmo\Mod{-Mod}
\dmo\CaMod{\Ca-Mod}
\dmo\Camod{\Ca-mod}
\dmo\Ind{Ind}
\dmo\Res{Res}
\dmo\Tab{Tab}
\dmo\STab{STab}
\dmo\Hom{Hom}
\dmo\Ext{Ext}
\dmo\End{End}
\dmo\Aut{Aut}
\dmo\So{S}
\dmo\coker{coker}
\dmo\im{im}
\dmo\de{depth}
\dmo\dwidth{\partial width}
\dmo\dec{depth^{class}}
\dmo\reg{reg}
\dmo\dreg{\partial reg}
\newtheorem{theorem}{Theorem}[section]
\newtheorem*{thm}{Theorem}
\newtheorem{thmab}{Theorem}

\newtheorem{corollary}[theorem]{Corollary}
\newtheorem{lemma}[theorem]{Lemma}
\newtheorem{proposition}[theorem]{Proposition}
\theoremstyle{definition}
\newtheorem{definition}[theorem]{Definition}
\newtheorem{remark}[theorem]{Remark}

\renewenvironment{proof}[1][\proofname]{{\bfseries #1\\}}{\qed}

\title{Homological Invariants of $\FI$-Modules and $\FI_G$-Modules}
\date{}
\author{Eric Ramos}

\address{University of Wisconsin - Madison Department of Mathematics, 480 Lincoln Dr., Madison, WI 53706-1325}
\email{eramos@math.wisc.edu}

\thanks{The author was supported by NSF-RTG grant 1502553.}

\begin{document}
\maketitle

\begin{abstract}
We explore a theory of depth for finitely generated $\FI_G$-modules. Using this theory, we prove results about the regularity, and provide novel bounds on stable ranges of $\FI$-modules, making effective a theorem of Nagpal and thereby refining the stable range in results of Church, Ellenberg, and Farb.
\end{abstract}

\section{Introduction}

Let $G$ be a group. The category $\FI_G$, introduced in \cite{SS2}, is that whose objects are finite sets, and whose morphisms are pairs $(f,g):S \rightarrow T$ such that $f:S \rightarrow T$ is an injection, and $g:S \rightarrow G$ is a map of sets. If $G = 1$ is the trivial group, then $\FI_G$ is equivalent to the category $\FI$ of finite sets and injections. The full subcategory of $\FI_G$ whose objects are the sets $[n] := \{1,\ldots,n\}$ is equivalent to $\FI_G$, and we therefore often identify the two. For any commutative ring $k$, an $\FI_G$-module over $k$ is a covariant functor $V:\FI_G \rightarrow \text{Mod}_k$. We will often write $V_n := V([n])$.\\

In the present paper we study various homological invariants of $\FI_G$-modules, and show how they relate to concrete questions about stability. In particular, we generalize the bounds on Castelnuovo-Mumford regularity in \cite[Theorem A]{CE}, and provide explicit bounds on results from \cite[Theorem B]{CEFN} and \cite{NS}. If $V$ is an $\FI_G$-module, then we define $H_0(V)$ on any finite set $[n]$ to be the quotient of $V_n$ by the images of all maps $V(f,g)$ where $(f,g) \in \Hom_{\FI_G}([m],[n])$ and $m < n$. The functor $V \mapsto H_0(V)$ is right exact, and we define its right derived functors, $H_i$, to be the \textbf{homology functors}. The paper \cite{CE} studied these functors in the case of $\FI$-modules, and showed various applications to the homology of congruence subgroups. Subsequently, Calegari and Emerton used the results of Church and Ellenberg in studying the homology of arithmetic groups \cite[Theorem 5.2]{CaE}.\\

After reviewing some preliminary topics, we next turn our attention to bounding the regularity of $\FI_G$-modules. We define the \textbf{degree} of an $\FI_G$-module to be the largest integer $n$ for which $V_n \neq 0$, while $V_r = 0$ whenever $r>n$. We say that $V$ is \textbf{generated in degree $\leq m$} if $\deg(H_0(V)) \leq m$ (See Definition \ref{fg}). Similarly, we say that $V$ has \textbf{first homological degree $\leq r$} if $\deg(H_1(V)) \leq r$ (See Definition \ref{fg}, and Remark \ref{reldegdef} for more on this definition). If a module has finite generating and relation degrees, then it is said to be \textbf{presented in finite degree}. The \textbf{regularity} of $V$ is the smallest integer $N$ such that
\[
\deg(H_i(V))-i \leq N
\]
for all $i \geq 1$ (see Definition \ref{reldeg}). It was proven by Sam and Snowden in \cite[Corollary 6.3.5]{SS3} that finitely generated $\FI$-modules in characteristic zero have finite regularity. Following this, Church and Ellenberg proved that $\FI$-modules which are presented in finite degree have finite regularity over any ring, and they provided a bound on this regularity \cite[Theorem A]{CE}. More recently, Li and Yu have provided different bounds on the regularity of $\FI$-modules \cite[Theorem 1.8]{LY}. One of the goals of this paper is to prove that similar bounds exist for $\FI_G$-modules. Indeed, we will find that the regularity of a module $V$, which is presented in finite degree, can be bounded in terms of the generating degree and the relation degree of $V$. \\

\begin{thmab}\label{finitereg}
Let $V$ be an $\FI_G$-module over a commutative ring $k$ with generating degree $\leq d$ and first homological degree $\leq r$. Then $\reg(V) \leq r + \min\{r,d\} -1$.\\
\end{thmab}

Note that the above bound exactly agrees with the bound given by Church and Ellenberg for $\FI$-modules in \cite[Theorem A]{CE}. Indeed, much of the work done in proving the above theorem will involve generalizing the techniques used in that paper.\\

The main tool in the first part of the paper is the use of the derivative functor. Given an $\FI_G$-module, we define its \textbf{first shift} $\So V$ to be the module defined on points by $\So V_n = V_{n+1}$. For any $(f,g):[n] \rightarrow [m]$, the map $\So V(f,g):V_{n+1} \rightarrow V_{m+1}$ will be the map $V(f_+,g_+)$, where $f_+$ agrees with $f$ on $[n]$ and maps $n+1$ to $r+1$, while $g_+$ agrees with $G$ on $[n]$ and maps $n+1$ to the identity (see Definition \ref{shiftfunctor}). We will write $\So_b$ to denote the $b$-th iterate of $\So$. This functor was first introduced in \cite[Definition 2.8]{CEFN}, and has since seen use in many papers (e.g. \cite{N}, \cite{NS}, \cite{L}, \cite{GL}). If $V$ is any $\FI_G$-module, then the map induced by the natural inclusion $f^n:[n] \rightarrow [n+1]$ (i.e. that which sends $i$ to $i$ for all $i$), paired with the trivial map into $G$, induces a map of $\FI_G$-modules $V \rightarrow \So V$. The \textbf{derivative} of $V$, denoted $DV$, is defined to be the cokernel of this map (see Definition \ref{dervdef}). As with the shift functor, we set $D^a$ to be the $a$-th iterate of $D$. Because this functor is right exact, we can consider its left derived functors, which we denote $H_i^{D^a}$.\\

The connection between the derivative functor and homology was established for $\FI$-modules in \cite{CE}. Church and Ellenberg show that many properties of the derivative is encoded in the combinatorics of $\FI$, which they then compute and relate to the regularity. One of the main objectives of the latter part of this paper is to argue that, in fact, the homological properties of the derivative functor provide deeper insights than were previously noted.\\

If $V$ is an $\FI_G$-module, then we set its \textbf{depth} to be the smallest non-negative value $a$ for which $H_1^{D^{a+1}}(V) \neq 0$ . While this definition does not seem related to more classical notions of depth, we will find that it satisfies many desirable properties. This is explored deeply in Section \ref{depthclass}. On the other hand, we define the \textbf{derived regularity} of $V$, denoted $\dreg(V)$, to be the maximum across all integers $a$ of the degree of $H_1^{D^a}(V)$. Using the properties of depth discussed in Section \ref{depthclass}, we will be able to prove the following theorem.\\

\begin{thmab}\label{homacyclic}
Let $V$ be a $\FI_G$-module which is generated in finite degree over a commutative ring $k$. Then the following are equivalent:
\begin{enumerate}
\item $V$ admits a filtration $0 = V^{(0)} \subseteq V^{(1)} \subseteq \ldots \subseteq V^{(n)} = V$, such that the cofactors are relatively projective (see Definition \ref{relproj});
\item There is a series of surjections $Q^{(n)} = V \twoheadrightarrow Q^{(n-1)} \twoheadrightarrow \ldots \twoheadrightarrow Q^{(0)} = 0$ whose successive kernels are relatively projective;
\item $V$ is homology acyclic;
\item $H_1(V) = 0$;
\item $V$ admits a finite resolution by homology acyclic objects which are generated in finite degree.\\
\end{enumerate}
If, in addition, $V$ is presented in finite degree, then the condition $H_i(V) = 0$ for some $i > 0$ is also equivalent to the above.
\end{thmab}

\begin{remark}
A very recent preprint of Li and Yu \cite{LY} has overlap with this paper. Namely, they prove a weaker version of Theorem \ref{homacyclic} as their Theorem 1.3. While earlier versions of this work were otherwise independent, we use arguments inspired by \cite[Section 3]{LY} in Section \ref{depthclass} to generalize the results of these previous versions.\\
\end{remark}

Relatively projective $\FI_G$-modules will be defined and expanded upon in later sections. For now, one can imagine these objects as being projective in the traditional sense. In fact, if $k$ is a field of characteristic zero, then projective and relatively projective will exactly coincide. Over a more general ring relatively projective $\FI_G$-modules need not be projective. However, they turn out to be acyclic with respect to many natural functors on the category of $\FI_G$-modules. If $k$ is a field of characteristic zero, then the above theorem implies that every non-projective object in the category of finitely generated $\FI_G$-modules has infinite projective dimension. This fact was proven by Sam and Snowden in the case of $\FI$-modules in \cite[Section 0.1]{SS3}. We say that an $\FI_G$ module $V$ is \textbf{$\sharp$-filtered} whenever it satisfies any of the conditions in the above theorem.\\

As a second application of depth, we will make effective, and generalize the following theorem of Nagpal and Snowden.\\

\begin{thm}[\cite{NS}]
Assume that $G$ is a polycyclic-by-finite group, and let $V$ be a finitely generated (see Definition \ref{fg}) $\FI_G$-module over a Noetherian ring $k$. Then for $b \gg 0$, $\So_b V$ is $\sharp$-filtered.\\
\end{thm}

We call the smallest $b \geq 0$ such that $\So_b V$ is $\sharp$-filtered the \textbf{Nagpal number} of $V$, $N(V)$. Nagpal proved the above theorem in the case where $G$ is finite in \cite[Theorem A]{N}.\\

\begin{thmab}\label{shiftreg}
Let $V$ be an $\FI_G$-module which is presented in finite degree over a commutative ring $k$. Then $\So_b V$ is $\sharp$-filtered for $b \gg 0$. Moreover, in this case $\So_b V$ is $\sharp$-filtered if and only if $b >\dreg(V)$. In particular, if $V$ is not $\sharp$-filtered then $N(V) = \dreg(V) + 1$.\\
\end{thmab}

We will see in Section \ref{derv} that the derived regularity of $V$ is related to the regularity of $V$. Theorem \ref{shiftreg} therefore provides a bridge between the phenomenon of eventual $\sharp$-filtration of a module and its regularity. Namely, if $V$ has first homological degree $\leq r$, then
\[
\reg(V) \leq N(V)+ r -2 \text{ (see Proposition \ref{homreg}.)}
\]
As an application of our results, we prove a theorem about bounding stable ranges of $\FI_G$-modules whenever $G$ is finite.\\

If $G$ is a finite group, and $V$ is a finitely generated $\FI_G$-module over a field $k$, we define its \textbf{Hilbert function} to be $H_V(n) = \dim_k(V_n)$. In \cite[Theorem 3.3.4]{CEF}, Church, Ellenberg, and Farb prove that if $k$ is a field of characteristic 0, and $V$ is a finitely generated $\FI$-module, then there is a polynomial $P_V(x) \in \Q[x]$ such that $H_V(n) = P_V(n)$ for all $n \gg 0$. They go on to show that this equality holds for $n \geq r+d$, where $r$ is the first homological degree of $V$, and $d$ is the degree of $H_0(V)$. This was was also proven by Sam and Snowden in \cite{SS3}, although their bound is stated in terms of a kind of local cohomology theory \cite[Theorem 5.1.3 and Remark 7.4.6]{SS3}. Later, Church, Ellenberg, Farb, and Nagpal \cite[Theorem B]{CEFN} prove that if $k$ is any field, then $H_V(n)$ agrees with a polynomial for $n \gg 0$. In this case, the authors do not provide bounds on when this stabilization occurs. These same theorems were later proven by Wilson in the case where $G = \Z/2\Z$ \cite[Theorems 4.16 and 4.20]{JW}. Later, Sam and Snowden proved that the Hilbert function is eventually polynomial for an arbitrary finite group, although they did not provide bounds on when the equality begins \cite[Theorem 10.1.2]{SS}.\\

The question of how big $n$ has to be before this stability begins is known as the stable range problem. We say that a finitely generated $\FI_G$-module $V$ over a field $k$ has \textbf{stable range $\geq m$} if there is a polynomial $P_V(x) \in \Q[x]$ such that for any $n \geq m$, $H_V(n) = P_V(n)$.\\

\begin{thmab}\label{fistab}
Let $G$ be a finite group, and let $V$ be a finitely generated $\FI_G$-module over a field $k$. Then the stable range of $V$ is at least $r + \min\{r,d\}$ where $r$ is the first homological degree of $V$, and $d$ is the generating degree.\\
\end{thmab}

The work in this paper therefore provides a new proof of the bounds given in \cite[Theorem 3.3.4]{CEF} and \cite[Theorem 4.16]{JW}, while providing a novel bound in the cases where $k$ is a general field or where $G \neq 1, \Z/2\Z$.\\

\section*{Acknowledgments}
The author would like to give thanks to Rohit Nagpal, Steven Sam, and Andrew Snowden for numerous useful discussions relating to this work. The author would also like to thank Liping Li for conversations on his work in \cite{LY}, which helped in generalizing earlier versions of this work. Finally, the author would like to give special thanks to Jordan Ellenberg, whose insights greatly improved many aspects of this paper.\\

\section{$\FI_G$-Modules}

\subsection{Basic Definitions}

For the remainder of this paper we fix a commutative ring $k$, and a group $G$. We will use $[n]$ to denote the set $[n] = \{1,\ldots, n\}$. By convention, $[0] = \emptyset$.\\

\begin{definition}
We define the category $\FI_G$ to be that whose objects are finite sets, and whose morphisms are pairs $(f,g):S \rightarrow T$, of an injection of sets $f:S \hookrightarrow T$, and a map of sets $g:S \rightarrow G$. Given two composable morphisms in $\FI_G$, $(f,g), (f',g')$, we define $(f,g) \circ (f',g') = (f \circ f', h)$, where $h(x) = g'(x)\cdot g(f'(x))$.\\

Note that for any $n$, $\Aut_{\FI_G}([n]) = k[\Sn_n \wr G]$. For the remainder of this paper we shall write $G_n$ to denote the group $\Sn_n \wr G$.\\
\end{definition}

One immediately observes that the full subcategory of $\FI_G$ whose objects are the sets of the form $[n]$ is equivalent to $\FI_G$. For convenience of exposition, we will from this point on refer to this category as $\FI_G$. In the case where $G$ is the trivial group, one sees that the category $\FI_G$ is naturally equivalent to the category $\FI$ of finite sets and injections. If instead we specialize to $G = \Z/2\Z$, then $\FI_G$ is naturally equivalent to the category $\FI_{BC}$ discussed in \cite{JW}.\\

\begin{definition}
An \textbf{$\FI_G$-module} over $k$ is a covariant functor $V:\FI_G \rightarrow \text{Mod}_k$ from $\FI_G$ to the category of $k$-modules. We will often use the shorthand $V_n := V([n])$, and write $(f,g)_\as:V_n \rightarrow V_m$ to denote the map induced by an arrow $(f,g) \in \Hom_{\FI_G}([n],[m])$. The collection of morphisms $(f,g)_\as$ are known as the \textbf{induced maps} of $V$, while the maps $(f,g)_\as$, with $n < m$, are called the \textbf{transition maps} of $V$. The collection of $\FI_G$-modules over $k$, along with natural transformations, form a category, which we denote $\FI_G\Mod$.\\
\end{definition}

Many constructions from the category $\Mod_k$ will continue to work in $\FI_G\Mod$, so long as one applies the construction "point-wise." For example, there is a natural notion of direct sum of two $\FI_G$-modules $V, W$, where we set $(V \oplus W)_n = V_n \oplus W_n$. The induced maps of the sum are defined in the obvious way. One may similarly define point-wise notions of kernel and cokernel, which make $\FI_G\Mod$ an abelian category.\\

One should observe that for any fixed $n$, and any $\FI_G$-module $V$, the module $V_n$ carries the action of an $k[\Aut_{\FI_G}([n])] = k[G_n]$ module. One may therefore think of an $\FI_G$-module as a single object which encodes a collection of compatible $G_n$ representations, where the compatibility is given by the transition maps. This was the original motivation for Church, Ellenberg, and Farb \cite{CEF} studying $\FI$-modules and their relationship with Church and Farb's representation stability found in \cite{CF}.\\

\begin{definition}
We use $\FB_G$ to denote the subcategory of $\FI_G$ whose objects are the sets $[n]$, and whose morphisms are pairs $(f,g)$ such that $f$ is a bijection. An \textbf{$\FB_G$-module over $k$} is a functor $V:\FB_G \rightarrow \text{Mod}_k$. We denote the category of $\FB_G$-modules over $k$ by $\FB_G$-Mod.\\
\end{definition}

One can think of $\FB_G$-modules as sequences of $k[G_n]$-modules, with $n$ increasing. We see that $\FB_G$-Mod can be thought of as a subcategory of $\FI_G$-Mod, the subcategory of modules with trivial transition maps. Because of this, we will often use terms and definitions from the theory of $\FI_G$-modules when describing $\FB_G$-modules.\\

\begin{definition}\label{relproj}
For any non-negative integer $n$, we define the \textbf{free $\FI_G$-module of degree $n$} $M(n)$ by the following assignments: $M(n)_m := k[\Hom_{\FI_G}([n],[m])]$ is the free $k$-module spanned by vectors $\{e_{(f,g)}\}$ indexed by the members of $\Hom_{\FI_G}([n],[m])$, while induced maps act on the natural basis by composition. We will also refer to direct sums of free modules as being free.\\

If $W$ is a $k[G_n]$-module, then we define the \textbf{relatively projective} $\FI_G$-module over $W$ by the following assignments: $M(W)_m = W \otimes_{k[G_n]} k[\Hom_{\FI_G}([n],[m])]$, while induced maps act by composition in the second coordinate. More generally, if $W$ is an $\FB_G$-module, then the rule $M(W) = \bigoplus_{n \geq 0} M(W_n)$ makes $M$ into a functor from $\FB_G$-Mod to $\FI_G$-Mod. Modules in the image of this functor will also be referred to as relatively projective. We observe that $M(n) = M(k[G_n])$.\\
\end{definition}

\begin{remark}
Note that the terminology for the above definitions is not consistent in the literature. Relatively projective modules are the same as those denoted $\FI\sharp$-modules in \cite{CEF}, and those denoted free in \cite{CE}. Free modules are the same as those denoted principally projective in \cite{SS}.\\
\end{remark}

\begin{proposition}\label{yoneda}
If $W$ is a $k[G_n]$-module, and $V$ is any $\FI_G$-module, then
\begin{eqnarray}
\Hom_{\FI_G\text{-Mod}}(M(W),V) = \Hom_{G_n}(W,V_n). \label{projadj}
\end{eqnarray}
\end{proposition}

\begin{proof}
Given any map $\phi_n$ from the right hand side, we can extend it to a map $\phi$ of $\FI_G$-modules by just insisting it commute with transition maps. For example, for any $m > n$, the module $M(W)_m = W \otimes_{G_n} k[\Hom_{\FI_G}([n],[m])]$ is generated by pure tensors $w \otimes (f,g)$, where $w \in W$ and $(f,g) \in \Hom_{\FI_G}([n],[m])$. We therefore define
\[
\phi_m:M(W)_m \rightarrow V_m, \text{  } \phi_m(w \otimes (f,g)) := (f,g)_\as(\phi_n(w \otimes id)).
\]
One can quickly see that this defines a well defined morphism of $\FI_G$-modules.\\
\end{proof}

The adjunction (\ref{projadj}) immediately implies that $M(W)$ is projective whenever $W$ is a projective $k[G_n]$-module. In fact, we will see in section \ref{nak} that all projective modules are relatively projective. Observe that this implies that the free $\FI_G$-modules are actually projective, and therefore $\FI_G\Mod$ has sufficiently many projective objects.\\

Note that in the special case where $W = k[G_n]$, the adjunction (\ref{projadj}) becomes
\[
\Hom_{\FI_G\text{-Mod}}(M(n),V) = V_n.
\]
In other words, a map from the free object of degree $n$ is equivalent to a choice of an element of $V_n$. More precisely, the map sending a homomorphism $\phi:M(n) \rightarrow V$ to $\phi(id_{[n]})$ is an isomorphism.\\

We will prove other important properties of the $M$ functor in Section \ref{mfun}.

\begin{definition} \label{fg}
Given a non-negative integer $m$, we say that an $\FI_G$-module $V$ is \textbf{generated in degree $\leq m$} if there exists a surjection
\[
\bigoplus_{i \in I} M(n_i) \twoheadrightarrow V,
\]
where $I$ is some index set and $n_i \leq m$ for all $i \in I$. If the index set $I$ can be taken to be finite, then we say that $V$ is \textbf{finitely generated}. We denote the category of finitely generated $\FI_G$-modules by $\FI_G$-mod. By convention, the trivial $\FI_G$-module is said to be generated in degree $\leq -1$.\\

We say that $V$ has \textbf{relation degree $\leq r$} if there is an exact sequence
\[
0 \rightarrow K \rightarrow M \rightarrow V \rightarrow 0,
\]
with $M$ relatively projective, such that $K$ is generated in degree $\leq r$. An exact sequence of the above form is known as a \textbf{presentation} for the module $V$, and we say that $V$ is \textbf{presented in finite degree} if $V$ has finite relation and generating degrees.\\

The \textbf{$M$-functor} $M:\FB_G\Mod \rightarrow \FI_G\Mod$ is that given by
\[
M(W) = \bigoplus_i M(W_i).
\]
\text{}\\
\end{definition}

Let $V$ be an $\FI_G$-module, and let $S$ be any subset of $\sqcup_n V_n$. Then we define the span of $S$ to be the $\FI_G$-module defined on objects by 
\[
\text{span}_k(S)_m = \{w \in V_m \mid w = \sum_i \lambda_i(f_i,g_i)_\as(x_i) \text{ with $x_i \in S$, $\lambda_i \in k$, and $(f_i,g_i) \in \Hom_{\FI_G}([n_i],[m])$}\},
\]
with induced maps restricted from $V$. Then $V$ is generated in degree $\leq n$ if and only if $\text{span}_k(\sqcup_{i \leq n} V_i) = V$. From the remark about maps from free objects, one can also see that $V$ is finitely generated if and only if it is the span of a finite set of elements.\\

\begin{definition}
Given an $\FI_G$-module $V$ we define $\deg(V)$, the \textbf{degree} of $V$, to be the supremum $\sup\{n \mid V_n \neq 0\} \in \N \cup \{-\infty, \infty\}$, where we use the convention that the supremum of the empty set is $-\infty$. We say that $V$ has \textbf{finite degree} if and only if $\deg(V) < \infty$.\\
\end{definition}

It is an immediate consequence of the relevant definitions that all $\FB_G$-modules with finite generating degree have finite degree.\\

One very non-obvious fact about the category $\FI_G$-mod is that it can be abelian. While finite generation is clearly preserved by quotients, it is not obvious that submodules of finitely generated objects are also finitely generated. We have the following theorem, usually called the \textbf{Noetherian property}.\\

\begin{theorem}[{SS2}, Corollary 1.2.2]
If $k$ is a Noetherian ring and $G$ is a polycyclic-by-finite group, then the category $\FI_G$-mod is abelian. That is, submodules of finitely generated $\FI_G$-modules are also finitely generated.\\
\end{theorem}

Historically, the Noetherian property was proven for $\FI$-modules over a field of characteristic 0 in \cite[Theorem 1.3]{CEF} and independently by Snowden in \cite[Theorem 2.3]{S}, and over a general Noetherian ring in \cite[Theorem A]{CEFN}. The case $G = \Z/2\Z$ was proven in \cite[Theorem 4.21]{JW}. The paper \cite{SS2} proves the theorem for all polycyclic-by-finite groups $G$.\\

\begin{remark}
Note that the above theorem requires that the group $G$ be polycyclic-by-finite. In this paper we will not need this assumption on $G$. In particular, our results will be independent of the Noetherian property.\\
\end{remark}

\subsection{The Homology Functors and Nakayama's Lemma} \label{nak}

\begin{definition}
Let $V$ be an $\FI_G$-module, and let $n$ be a non-negative integer. We use $V_{<n} \subseteq V_n$ to denote the submodule of $V_n$ spanned by the images of all transition maps. Put another way, $V_{<n}$ is the submodule of $V_n$ generated by the elements 
\[
\cup_{i < n} \cup_{(f,g) \in \Hom_{\FI_G}([i],[n])}(f,g)_\as(V_i).
\]
The \textbf{zeroth homology} of $V$ is the $\FB_G$-module defined by $H_0(V)_n = V/V_{<n}$.\\
\end{definition}

This notion was first introduced for $\FI$-modules in \cite{CEF}, and later expanded upon in \cite{CEFN} and \cite{CE}. This functor was also considered in \cite{GL} and \cite{GL2}, albeit in a slightly different language.\\

\begin{proposition}\label{naklemma}
The zeroth homology functor $H_0$ enjoys the following properties:
\begin{enumerate}
\item for any $k[G_n]$-module $W$, $H_0(M(W))_n = W$, while $H_0(M(W))_m = 0$ for all $m \neq n$;
\item if $\{v_i\}_{i \in I} \subseteq \sqcup_n H_0(V)_n$ is a generating set for $H_0(V)$, and $w_i$ is a lift of $v_i$ for each $i \in I$, then $\{w_i\}_{i \in I}$ is a generating set for $V$. Equivalently, $H_0(V) = 0$ if and only if $V = 0$ (Nakayama's Lemma);
\item $H_0(V)$ is generated in degree $\leq n$ (resp. finitely generated) if and only if $V$ is generated in degree $\leq n$ (resp. finitely generated);
\item $H_0$ is left adjoint to the inclusion functor $\FB_G$-Mod $\rightarrow \FI_G$-Mod;
\item $H_0(V)$ is right exact, and maps projective modules to projective modules.\\
\end{enumerate}
\end{proposition}

\begin{proof}
The first non-zero entry in $M(W)$ is $M(W)_n = W$. On the other hand, one sees immediately from definition that $M(W)$ is generated in this degree. In other words, all other elements in $M(W)$ are linear combinations of transition maps applied to elements of $M(W)_n$. This implies the first statement.\\

Let $v_i$ and $w_i$ be as in the second statement. Let $j$ be the least index such that $V_j \neq 0$. Then $H_0(V)_j = V_j$, and therefore $V_j$ is generated by the $w_i$ by assumption. To finish the proof we proceed by induction. If $n > j$, and $v \in V_n$, then the image of $v$ in $H_0(V)_n$ can be expressed as a linear combination of the $v_i$. In particular, $v$ is a linear combination of the $w_i$, as well as images of elements from lesser degrees. Applying the inductive hypothesis completes the proof.\\

The third statement is an immediate consequence of Nakayama's Lemma.\\

Let $V$ be an $\FI_G$-module, and let $W$ be a $\FB_G$-module. If $\phi:H_0(V) \rightarrow W$ is any map of $\FB_G$-modules, then for each $n$ we define a map of $k[G_n]$-modules $\widetilde{\phi}_n:V_n \rightarrow W_n$ via $\widetilde{\phi}_n(v) = \phi_n(\pi(v))$, where $\pi:V_n \rightarrow H_0(V)_n$ is the quotient map. We claim that $\widetilde{\phi}$ is actually a morphism of $\FI_G$-modules. If $(f,g)_\as$ is any transition map, and $v \in V_m$, then
\[
\widetilde{\phi}_n((f,g)_\as(v)) = \phi_n(\pi((f,g)_\as(v))) = 0 = (f,g)_\as(\widetilde{\phi}_n(v)).
\]
On the other hand, if $\sigma \in k[G_n]$, then
\[
\widetilde{\phi}_n(\sigma(v)) = \phi_n(\pi(\sigma(v))) = \phi_n(\sigma(\pi(v))) = \sigma(\phi_n(\pi(v))) = \sigma(\widetilde{\phi}(v)).
\]
Conversely, let $\phi:V \rightarrow W$ be a morphism of $\FI_G$-modules. Because $\phi$ respects transition maps, and because $W$ has trivial transition maps, it follows that $\phi_n$ vanishes on the images of the transition maps into $V_n$. In particular, the map $\widetilde{\phi}:H_0(V) \rightarrow W$ given by $\widetilde{\phi}_n(v) = \phi_n(v)$ is well defined. The two constructions give above are clearly inverses of one another, proving the adjunction.\\

The last statement is a consequence of standard homological algebra. Left adjoints are always right exact, and any left adjoint to an exact functor must preserve projectives.\\

\end{proof}

As a quick application of the above proposition, we prove that all projective modules are relatively projective.\\

\begin{proposition}\label{projclass}
Let $V$ be an $\FI_G$-module. Then $V$ is projective if and only if $V = \bigoplus_i M(W_i)$, where each $W_i$ is some projective $k[G_i]$-module. In particular, all projective $\FI_G$-modules are relatively projective.\\
\end{proposition}

\begin{proof}
We have already seen that modules of the form $\bigoplus_i M(W_i)$, with $W_i$ projective, are projective. Conversely, let $V$ be a projective $\FI_G$-module. Then part 5 of Proposition \ref{naklemma} implies that $H_0(V)$ is a projective $\FB_G$-module. It follows that the quotient map $q_n:V_n \rightarrow H_0(V)_n$ admits a section $\iota_n:H_0(V)_n \rightarrow V_n$. Proposition \ref{yoneda} implies there is a map $\bigoplus_i M(H_0(V)_i) \rightarrow V$ induced by the collection of $\iota_n$. Nakayama's lemma implies that the map is a surjection. We claim that this map is actually injective as well. Let $K$ be its kernel, and apply $H_0$ to the exact sequence
\[
0 \rightarrow K \rightarrow \bigoplus M(H_0(V)_i) \rightarrow V \rightarrow 0.
\]
Because $V$ is projective, and because $H_0$ is right exact by the previous proposition, it follows that there is an exact sequence
\[
0 \rightarrow H_0(K) \rightarrow H_0(V) \rightarrow H_0(V) \rightarrow 0
\]
where we have used part 2 to simplify the second term. It is easy to see from construction that the right most non-trivial morphism in this sequence is an isomorphism, and therefore $H_0(K) = 0$. Nakayama's lemma now implies that $K = 0$, as desired.\\
\end{proof}

It follows from this proposition that if $k$ is a field of characteristic 0, then the notions of relatively projective and projective coincide. Over an arbitrary commutative ring, this is no longer the case. We will find that this has interesting consequences later.\\

\begin{definition}
We write $H_i$ for the $i$-th derived functor of $H_0$. We call the collection of these functors the \textbf{homology functors}.\\
\end{definition}

The following nomenclature is used in \cite{L}.\\

\begin{definition}\label{reldeg}
If $V$ is an $\FI_G$-module, then for each $i$ we define its \textbf{$i$-th homological degree} to be $hd_i(V) = \deg(H_i(V))$. For some non-negative constant $N$, we say that $V$ has \textbf{regularity $\leq N$} if $hd_i(V)-i \leq N$ for each $i \geq 1$. We write $\reg(V)$ for the smallest value $N$ for which $V$ has regularity $\leq N$.\\
\end{definition}

\begin{remark}\label{reldegdef}
Nakayama's lemma tells us that the zeroth homological degree is an optimal bound on the the generating degree of $V$. One would hope that the first homological degree would be an optimal bound on the relation degree of $K$. Indeed, if 
\[
0 \rightarrow K \rightarrow M \rightarrow V \rightarrow 0
\]
is a presentation of V with $M$ generated in degree $\leq d$ and $K$ generated in degree $\leq r$, then an application of the $H_0$ functor implies that
\begin{eqnarray}
hd_1(V) \leq hd_0(K) = r \leq \max\{d,hd_1(V)\}.\label{reldegbound}
\end{eqnarray}
Despite this, it is not clear whether $V$ admits a presentation whose kernel is generated in degree $\leq hd_1(V)$.\\

The bounds (\ref{reldegbound}) imply that $r = hd_1(V)$ whenever $hd_1(V) \geq hd_0(V)$. This is the typical case. In fact, if instead we assume that $hd_1(V) < hd_0(V)$, then Li and Yu show \cite[Corollary 3.4]{LY} there exists an exact sequence
\[
0 \rightarrow V' \rightarrow V \rightarrow Q \rightarrow 0
\]
where $Q$ is relatively projective and $hd_0(V') = hd_0(V)-1$. One immediately notes from this, and Theorem \ref{homacyclic}, that $H_i(V') = H_i(V)$ for $i \geq 1$. In the paper \cite[Theorem A]{CE}, Church and Ellenberg show that
\[
\reg(V) \leq r + \min\{r,d\} -1
\]
where $r$ is the relation degree of $V$ and $d$ is the generating degree. The observations made in this remark will allow us to convert this bound to a bound in terms of the first homological degree $hd_1(V)$. If it is the case that $hd_1(V) \geq d$, then $r = hd_1(V) \geq d$, and the above bound becomes
\[
\reg(V) \leq r + d-1 = hd_1(V) + \min{hd_1(V),d} - 1.
\]
Otherwise, we may apply the lemma of Li and Yu, as well as induction, to conclude there is some submodule $V'' \subseteq V$ which is generated in degree $\leq hd_1(V)$ and 
\[
\reg(V) = \reg(V'') \leq 2hd_1(V'')-1 = hd_1(V) + \min\{hd_1(V),d\} - 1.
\]
We may therefore conclude that the bounds of Church and Ellenberg remain true when the relation degree is replaced with the first homological degree. In this paper, we will prove bounds using methods of Church and Ellenberg. To conclude the bounds promised in the introduction, one simply applies the methods in this remark.\\
\end{remark}

The main result of \cite{CE} is a bound on the regularity of an $\FI$-module in terms of its generating and relation degrees. Later, \cite[Theorem 1.5]{L} used different methods to prove conditional bounds on the regularity of finitely generated $\FI_G$-modules whenever $G$ is finite. Li also gives non-conditional bounds in the case of $\FI$, and where $k$ is a field of characteristic 0 \cite[Theorem 1.17]{L}\\

\subsection{The Category $\FI_G\sharp$ and the $M$ functor} \label{mfun}

\begin{definition} \label{sharp}
We define the category $\FI_G\sharp$ as follows. The objects of the category $\FI_G\sharp$ are once again the sets $[n]$, while the morphisms are triples $(A,f,g):[n] \rightarrow [m]$ such that $A\subseteq [n]$, $f:A \rightarrow [m]$ is an injection, and $g:A \rightarrow G$ is a map of sets. Composition in this category is defined in the following way. If $(A,f,g)$ and $(B,f',g')$ are two morphisms which can be composed then $(B,f',g') \circ (A,f,g) = (A \cap f^{-1}(B), f' \circ f,h)$ where $h(x) = g(x)g'(f(x))$, as before. An $\FI_G\sharp$-module over $k$ is a covariant functor $\FI_G\sharp \rightarrow \text{Mod}_k$.
\end{definition}

This category has been studied in the case where $G$ is the trivial group \cite{CEF}, as well as the case where $G = \Z/2\Z$ \cite{JW}. One sees that there is a natural inclusion $\FI_G \rightarrow \FI_G\sharp$, which induces a forgetful functor $\FI_G\sharp\text{-Mod} \rightarrow \FI_G\text{-Mod}$. For this reason we may consider a kind of homology functor $H_0:\FI_G\sharp\text{-Mod} \rightarrow \FB_G\text{-Mod}$ which is defined as the composition of the forgetful map, and the usual zeroth homology functor.\\

The first example of an $\FI_G\sharp$-module is the free module $M(m)$. Indeed, We endow $M(m)$ with the structure of an $\FI_G\sharp$-module as follows. let $e_{(f,g)}\in M(m)_n$ be one of the canonical basis vectors, and let $(A,f',g'):[n] \rightarrow [r]$ be a morphism in $\FI_G\sharp$. Then we set
\[
(A,f',g')_\as e_{(f,g)} = \begin{cases} 0 &\text{ if $f([m]) \not\subseteq A$}\\ e_{(f' \circ f, h)} &\text{ otherwise,}\end{cases}
\]
where $h:[m] \rightarrow G$ is the function $h(x) = g(x)g'(f(x))$. This same argument shows that $M(W)$ is an $\FI_G\sharp$-module for any $\FB_G$-module $W$.\\

The above discussion shows that we may consider $M$ as being valued in $\FI_G\sharp$-Mod.\\

\begin{proposition}\label{mprop}
The functor $M:\FB_G$-Mod $\rightarrow \FI_G$-Mod enjoys the following properties
\begin{enumerate}
\item The composition $H_0 \circ M$ is isomorphic to the identity;
\item $M$ is exact;
\item for all $i \geq 1$, $H_i \circ M = 0$.\\
\end{enumerate}
\end{proposition}

\begin{proof}
The first statement follows immediately from the first part of Proposition \ref{naklemma} and the definition of $M$.\\

For the second statement, it suffices to show that the functor preserves exactness of sequences of the form
\[
0 \rightarrow W' \rightarrow W \rightarrow W'' \rightarrow 0
\]
where $W',W$ and $W''$ are $k[G_m]$-modules for some $m$. For any $n$, we have that $M(W)_n = W \otimes_{k[G_m]} k[\Hom_{\FI_G}([m],[n])]$. Because kernels and cokernels are computed point-wise, it suffices to show that $k[\Hom_{\FI_G}([m],[n])]$ is a flat $k[G_m]$-module. Fix a representative from each orbit of the $G_m$ action on $\Hom_{\FI_G}([m],[n])$. If we set $I$ to be the collection of these maps, then let $\mathfrak{B} = \{e_{(f,g)}\}_{(f,g) \in I}$ be the associated set of canonical basis vectors of $k[\Hom_{\FI_G}([m],[n])]$. We claim that $k[\Hom_{\FI_G}([m],[n])]$ is a free $k[G_m]$-module with basis $\mathfrak{B}$. Because the orbits partition the whole of $\Hom_{\FI_G}([m],[n])$, it follows that this set is spanning. On the other hand, assume that one has an equation $\sum_{(f,g) \in \mathfrak{B}} e_{(f,g)} x_{(f,g)} = 0$, for some $x_{(f,g)} \in k[G_m]$. We may write $x_{(f,g)} = \sum_{\sigma \in G_m} a_{(f,g),\sigma}\sigma$, and therefore
\[
\sum_{(f,g),\sigma} a_{(f,g),\sigma}e_{(f,g)\circ\sigma} = 0.
\]
We observe that for distinct $(f,g),(f',g') \in \mathfrak{B}$ and any $\sigma,\tau \in G_m$, the elements $(f,g)\circ\sigma$ and $(f',g')\circ\tau$ must be distinct, as they are in different orbits by construction. If we fix $(f,g)$ and vary $\sigma$, then $(f,g)\circ\sigma = (f,g)\circ\tau$ implies that $\sigma = \tau$ because $(f,g)$ is monic. In particular, the above sum can be written
\[
\sum_{(f,g) \circ \sigma = (f',g') \in \Hom_{\FI_G}([m],[n])} a_{(f,g),\sigma}e_{(f',g')} = 0
\]
with each $(f',g')$ appearing at most once. This implies that $a_{(f,g),\sigma} = 0$ for all $f,g,$ and $\sigma$, as desired.\\

The final statement follows from the first two. Because $M$ maps projective objects to projective objects by Proposition \ref{projclass}, The derived functor of the composition $H_0 \circ M$ can be computed using the Grothendieck spectral sequence. This spectral sequence will only have one row because $M$ is exact. It therefore degenerates and we find that the derived functors of $H_0 \circ M$ are isomorphic to $H_i \circ M$. On the other hand, statement 2 tells us that $H_0 \circ M$ is the identity functor, which is clearly exact. This completes the proof.\\
\end{proof}

We note that the composition $M \circ H_0$ is not isomorphic the identity functor if we consider $M$ as being valued in $\FI_G$-Mod. If we instead consider $M$ as being valued in $\FI_G\sharp$-Mod, then this composition is isomorphic to the identity as the following theorem shows.\\

\begin{theorem}[\cite{CEF},\cite{JW}]\label{equiv}
The functor $M:\FB_G\text{-Mod} \rightarrow \FI_G\sharp\text{-Mod}$ is an equivalence of categories with inverse $H_0:\FI_G\sharp\text{-Mod} \rightarrow \FB_G\text{-Mod}$.\\
\end{theorem}

The two citations given prove the theorem in the cases where $G$ is the trivial group, and where $G = \Z/2\Z$, respectively. The proofs go through essentially word for word to prove Theorem \ref{equiv} in the general case.\\ 

Theorem \ref{equiv} can be considered the justification for the terminology $\sharp$-filtered from Theorem \ref{homacyclic}.\\

\subsection{The Shift Functor and Torsion}

The final piece we need from the basic theory of $\FI_G$-modules is the shift functor. This functor was heavily featured in both \cite{GL} and \cite{L}, and will be of great use to us in what follows.\\

\begin{definition}\label{shiftfunctor}
Let $\Sigma$ denote the endofunctor of $\FI_G$, which sends $[n]$ to $[n+1]$, and takes a map $(f,g):[n] \rightarrow [m]$ to the map $(f_{+},g_{+}):[n+1] \rightarrow [m+1]$ defined by
\[
f_{+}(x) = \begin{cases} f(x) & \text{ if $x \neq n+1$}\\ m+1 & \text{ otherwise} \end{cases} \hspace{1cm} g_{+}(x) = \begin{cases} g(x) & \text{ if $x \neq n+1$}\\ 1 & \text{ otherwise.} \end{cases}
\]
We define the \textbf{shift functor $\So$} with respect to $\Sigma$ to be the endofunctor of $\FI_G$-Mod
\[
\So V := V \circ \Sigma.
\]
For any integer $b \geq 1$ we set $\So_b$ to be the $b$-th iterate of $\So$.\\
\end{definition}

Shift functors were originally introduced in \cite{CEFN} in the case of $\FI$-modules, and have since seen use in various papers in the field (e.g. \cite{N}, \cite{GL}, \cite{NS}, \cite{L}). The following proposition collects many of the important properties of the shift functor.\\

\begin{proposition}\label{shiftprop}
The shift functor $\So$ enjoys the following properties:
\begin{enumerate}
\item $\So$ is exact;
\item if $V$ is generated in degree $\leq n$, then so is $\So V$;
\item If $W$ is any $k[G_n]$-module, then $\So(M(W)) = M(\Res_{k[G_{n-1}]}^{k[G_n]} W) \oplus M(W)$.\\
\end{enumerate}
\end{proposition}

\begin{proof}
Kernels and cokernels are computed point-wise by definition. It follows immediately from this that $S$ is exact.\\

If $0 \rightarrow K \rightarrow F \rightarrow V \rightarrow 0$ is a presentation for $V$, with $F$ free, then exactness of the shift functor implies that $0 \rightarrow \So K \rightarrow \So F \rightarrow \So V \rightarrow 0$ is exact as well. It therefore suffices to show that $\So(M(m))$ is generated in degree $\leq m$. Let $n > m+1$, and let $e_{(f,g)}$ be a canonical basis vector in $M(m)_n$. Let $h:[m] \rightarrow [m+1]$ be the injection which sends $f^{-1}(n)$ to $m+1$, if it exists, and is the identity elsewhere, and let $\widetilde{h}:[m] \rightarrow [n-1]$ be the injection which agrees with $f$ away from $f^{-1}(n)$, and sends $f^{-1}(n)$ to something outside the image of $f$. Finally, we let $\mathbf{1}:[m] \rightarrow G$ be the trivial map into $G$. Then we have
\[
\Sigma(\widetilde{h},\mathbf{1}) \circ (h,g) = (f,g)
\]
This shows that $\So M(m)$ is generated in degree $m$, as desired.\\

For the first part of final statement, Theorem \ref{equiv} implies that to show that $\So M(W)$ is relatively projective, it will suffice to show that it is an $\FI_G\sharp$-module. Let $(A,f,g):[m] \rightarrow [n]$ be a morphism in $\FI_G\sharp$. Then we may define an endofunctor $\Sigma_\sharp$ of $\FI_G\sharp$, which maps $[m]$ to $[m+1]$ and $\Sigma_\sharp(A,f,g) = (A \cup \{m+1\},f_+,g_+)$ where $f_+$ agrees with $f$ on $A$, and sends $m+1$ to $n+1$, and $g_+$ agrees with $g$ on $A$ and sends $m+1$ to the identity. Observe that $\Sigma_\sharp$ restricts to $\Sigma$ on $\FI_G \subseteq \FI_G\sharp$. In particular, the functor $\So$ can be extended naturally to a functor on $\FI_G\sharp$. This shows that $\So M(W)$ is relatively projective.\\

Once again applying Theorem \ref{equiv}, it remains to compute $H_0(\So M(W))$. The second part of this proposition implies that it suffices to compute $H_0(\So M(W))$ in degrees $m-1$ and $m$. It is clear from definition that $H_0(\So M(W))_{m-1} = \Res_{k[G_{m-1}]}^{k[G_m]} W$. A direct computation also shows that the transition maps originating from $\So M(W)_{m-1}$ will hit all pure tensors in $\So M(W)_m = M(W)_{m+1}$ except for those of the form $w \otimes (f,g)$ where $f^{-1}(m+1) = \emptyset$. The group $G_m$ now acts on these pure tensors in precisely the way it acts on $W$. In particular, $H_0(W)_m = W$, which concludes the proof.\\
\end{proof}

\begin{remark}
If $G$ is an infinite group, then shifts do not need to preserve finite generation. Indeed,
\[
\So M(m) = M(m-1)^{m \cdot |G|} \oplus M(m)
\]
by the above proposition.\\
\end{remark}

Note that the last two properties were proven for $\FI$-modules in \cite[Lemma 2.12]{CEFN} and \cite[Lemma 2.2]{N}. The next property of the shift functor which is important to us is its connection with torsion.\\

\begin{definition}
Let $V$ be an $\FI_G$-module. Fix $b \geq 0$, and let $(f_b^n,\mathbf{1}):[n] \rightarrow [n+b]$ denote the morphism in $\FI_G$ whose injection is the standard inclusion ($j$ maps to $j$ for all $j$), and whose $G$-map is the trivial map. Then the collection of the induced maps $(f^n_b,\mathbf{1})_\as:V_n \rightarrow V_{n+b}$ define a morphism of $\FI_G$-modules $\iota_b:V \rightarrow \So_b V$. We say that $V$ is \textbf{torsion free} if $\iota_b$ is injective for all $b$. Any element of $\sqcup_n V_n$ which appears in the kernel of some $\iota_b$ is called a \textbf{torsion} element of $V$. If every element of $V$ is torsion, then we say the module $V$ is itself torsion.\\
\end{definition}

The fact that every $\FI_G$-module maps into its shift will be used throughout this paper. One should observe that an $\FI_G$-module $V$ is torsion free if and only if $\iota := \iota_1$ is injective. Indeed, if $v \in V_n$ is in the kernel of some $(f_b^n,\mathbf{1})_\as$, then we write
\[
0 = (f_b^n,\mathbf{1})_\as(v) = (f_1^{n+b-1},\mathbf{1})_\as(f_{b-1}^{n},\mathbf{1})_\as(v)
\]
If $(f_{b-1}^{n},\mathbf{1})_\as(v) = 0$, then we repeat the above until we find a non-trivial element in the kernel of $(f_1^{a},\mathbf{1})_\as$ for some $a \geq n$. Also note that if $v \in V_n$ is in the kernel of some transition map, then it must in fact be in the kernel of some $\iota$ as well. Indeed, this follows from the fact that the action of $G_n$ on $\Hom_{\FI_G}([m],[n])$ is transitive.\\

\begin{lemma}\label{inddeg}
Let $V$ be an $\FI_G$-module, which is generated in degree $\leq m$ and related in degree $\leq r$. Then for any $b$, $\coker(V \rightarrow \So_b V)$ is generated in degree $< m$ and related in degree $< r$.\\
\end{lemma}

\begin{proof}
Looking through the proof of the second part of Proposition \ref{shiftprop}, one finds that the inclusion
\[
M(W) \hookrightarrow M(\Res_{k[G_{n-1}]}^{k[G_n]} W) \oplus M(W) = \So M(W)
\]
is exactly $\iota$. Let $F$ be a free module generated in degree $\leq m$ which surjects onto $V$. Then exactness of the shift functor implies we have the following commutative diagram with exact rows,
\[
\begin{CD}
\coker(K \rightarrow \So_b K) @>>> \coker(F \rightarrow \So_b F) @>>> \coker(V \rightarrow \So_b V)  @>>> 0\\
@AAA @AAA    @AAA @.\\
\So_b K @>>>\So_b F @>>> \So_b V  @>>> 0\\
@AAA @AAA @AAA @.\\
K @>>> F @>>> V  @>>> 0\\
\end{CD}
\]
The middle column is a split exact sequence $0 \rightarrow F \rightarrow Q \oplus F \rightarrow Q \rightarrow 0$, for some free module $Q$ generated in degree $< m$, by the previous remarks. This shows that $\coker(V \rightarrow \So_b V)$ is generated in degree $< m$. Therefore, $\coker(K \rightarrow \So_b K)$ is generated in degree $< r$. Because the rows of the above diagram are exact, we conclude that the relation degree of $\coker(V \rightarrow \So_b V)$ is $< r$.\\
\end{proof}

\subsection{$\sharp$-Filtered Objects and the First Half of Theorem \ref{homacyclic}}\label{sfil}

\begin{definition}\label{sfildef}
We say that an $\FI_G$-module $V$ is \textbf{$\sharp$-filtered} if it admits a filtration
\[
0 = V^{(0)} \subseteq \ldots \subseteq V^{(n-1)} \subseteq V^{(n)} = V
\]
whose cofactors are relatively projective.\\
\end{definition}

If $k$ is a field, and $G$ is a finite group, the dimension data of a finitely generated $\sharp$-filtered object is described by a single polynomial for all $n$. That is to say, the Hilbert function
\[
n \mapsto \dim_k V_n
\]
is a polynomial in $n$ for all $n$. Indeed, a direct computation verifies that for any finite dimensional $k[G_m]$-module $W$, 
\[
\dim_k M(W)_n = \binom{n}{m} \dim_k W 
\] 
for all $n \geq 0$.

\begin{theorem}[\cite{NS}]\label{polystab}
Assume that $G$ is a polycyclic-by-finite group, and let $V$ be a finitely generated $\FI_G$-module over a Noetherian ring $k$. Then for $b \gg 0$, $\So_b V$ is $\sharp$-filtered.\\
\end{theorem}

\begin{definition}
The \textbf{Nagpal number}, $N(V) \in \N \cup \{\infty\}$, of an $\FI_G$-module $V$ is the smallest value $b$ such that $\So_bV$ is $\sharp$-filtered.\\
\end{definition}

Note that in the context of this paper, it is not clear whether $N(V)$ is finite. The Nagpal-Snowden theorem tells us that this will be the case whenever $G$ is polycyclic-by-finite, $V$ is finitely generated, and $k$ is a Noetherian ring. One of the main results of this paper will be to show that $N(V)$ is finite in many other cases as well (see Theorem \ref{shiftreg}). The above discussion implies the following immediate corollary.\\

\begin{corollary}
If $G$ is a finite group, and $V$ is a finitely generated $\FI_G$-module over a field $k$, then there is a polynomial $P_V(x) \in \Q[x]$ such that the Hilbert function $H_V(n) = \dim_k V_n$ is equal to $P_V(n)$ for $n \geq N(V)$.\\
\end{corollary}

\begin{remark}
Keeping in mind Theorems \ref{shiftreg} and \ref{fistab}, the above corollary provides a parallel between $\FI_G$-modules and graded modules over a polynomial ring. Namely, it is a consequence of the Hilbert Syzygy Theorem that the regularity of a graded module $M$ over a polynomial ring provides a bound to the obstruction of the Hilbert polynomial (See \cite{E} or \cite{E2}). The results of this paper therefore imply a similar relationship between the regularity of an $\FI_G$-module and bounds on its stable range. This might come as somewhat of a surprise, as Theorem \ref{homacyclic} implies that all non-$\sharp$-filtered modules require infinite resolutions; a stark contrast to the Hilbert Syzygy Theorem.\\
\end{remark}

This corollary was proven for $\FI$-modules over a field of characteristic 0 in \cite[Theorem 1.5]{CEF} and \cite[Theorem 5.1.3]{SS3}, and over an arbitrary field in \cite[Theorem B]{CEFN}. Following this, polynomial stability was proven in the case where $G = \Z/2\Z$ in \cite[Theorem 4.20]{JW}. It was proven for general $\FI_G$-modules in \cite[Theorem 10.1.2]{SS}. None of these sources used the Nagpal-Snowden Theorem in their work. Using the new homological invariants defined in this paper, we will be able to replace $n \geq N(V)$ in the above corollary with an explicit lower bound on $n$.\\

\begin{remark}
Although it is not proven in \cite{NS}, Theorem \ref{polystab} actually implies that the Grothendieck group $\mathcal{K}_0(\FI_G\text{-mod})$ is generated by the classes of torsion modules and relatively projective modules whenever $k$ is a Noetherian ring and $G$ is polycyclic-by-finite. Indeed, if $V$ is an $\FI_G$-module we have the exact sequence
\[
0 \rightarrow T(V) \rightarrow V \rightarrow V' \rightarrow 0
\]
where $V'$ is torsion free. Because $V'$ is torsion free, it embeds into all of its shifts. The Nagpal-Snowden theorem therefore implies that $V'$ embeds into a $\sharp$-filtered object, and Lemma \ref{inddeg} shows that the cokernel of this embedding is generated in strictly lower degree than $V$. Induction implies the desired result. This fact was proven for $\FI$-modules over a field of characteristic 0 by Sam and Snowden in \cite[Proposition 4.9.1]{SS3}. Note that we may also view this presentation of the Grothendieck group as a consequence of the classification theorem from Section \ref{depthclass}.\\
\end{remark}

At this point in the paper, we are ready to prove the first collection of equivalences guaranteed by Theorem \ref{homacyclic}.\\

\begin{theorem}\label{whomacyclic}
For an $\FI_G$-module $V$ which is generated in finite degree, the following are equivalent:
\begin{enumerate}
\item $V$ is $\sharp$-filtered;
\item There is a series of surjections $Q^{(n)} = V \twoheadrightarrow Q^{(n-1)} \twoheadrightarrow \ldots \twoheadrightarrow Q^{(0)} = 0$ whose successive kernels are relatively projective;
\item $V$ is homology acyclic;
\item $H_1(V) = 0$;\\
\end{enumerate}
\end{theorem}

\begin{proof}
The third part of Proposition \ref{mprop} shows that the first two statement imply the third, and clearly the third implies the fourth.\\

Assume that $H_1(V) = 0$, and let $i$ be the least index such that $V_i \neq 0$. Then we may construct a map $M(V_i) \rightarrow V$, which is an isomorphism in degree $i$. Denote the kernel of this map by $K^{(n)}$, and its image by $I^{(n)}$. This leaves us with a pair of exact sequences,
\begin{eqnarray}
0 \rightarrow K^{(n)} \rightarrow M(V_i) \rightarrow I^{(n)} \rightarrow 0\label{hom1}\\ 
0 \rightarrow I^{(n)} \rightarrow V \rightarrow Q^{(n-1)} \rightarrow 0\label{hom2}
\end{eqnarray}
Applying $H_0$ to (\ref{hom1}), we find that $H_0(M(V_i))$ surjects onto $H_0(I^{(n)})$ and therefore these must be isomorphic. Indeed, $H_0(M(V_i))$ is zero everywhere but in degree $i$, where it is $V_i$, and $H_0(I^{(n)})$ must be $V_i$ in degree $i$ by construction. This shows that the map $H_0(I^{(n)}) \rightarrow H_0(V)$ is an injection. Applying $H_0$ to (\ref{hom2}) and using our assumption we obtain the exact sequence
\begin{eqnarray}
H_2(Q^{(n-1)}) \rightarrow H_1(I^{(n)}) \rightarrow 0 \rightarrow H_1(Q^{(n-1)}) \rightarrow H_0(I^{(n)}) \rightarrow H_0(V) \label{hom3}
\end{eqnarray}
By what was just discussed we may conclude that $H_1(Q^{(n-1)}) = 0$. We observe that the first degree $j$ for which $Q^{(n-1)}_j \neq 0$ will be strictly larger than $i$. This allows us to iterate the above process. Moreover, because $V$ was generated in finite degree, we know that the same is true about $Q$. This shows that $H_0(Q)$ is supported in precisely one less degree than $H_0(V)$. It follows from this that this process will eventually terminate. To finish the proof, it suffices by induction to show that $K^{(n)} = 0$. Once again looking at $H_0$ applied to (\ref{hom1}) we find
\[
0 \rightarrow H_1(I^{(n)}) \rightarrow H_0(K^{(n)}) \rightarrow H_0(M(V_i)) \rightarrow H_0(I^{(n)}) \rightarrow 0
\]
We have already discussed that the last map is an isomorphism, so $H_1(I^{(n)}) = 0$ if and only if $H_0(K^{(n)}) = 0$. In this case Nakayama's lemma would imply that $K^{(n)} = 0$. It therefore remains to show that $H_1(I^{(n)}) = 0$.\\

Note that at the final step in this construction we will be left with a sequence of the form
\[
0 \rightarrow K^{(1)} \rightarrow M(W') \rightarrow Q^{(1)} \rightarrow 0,
\]
where $W'$ is a $k[G_j]$-module for some $j$. Indeed, $H_0(Q^{(1)})$ is only supported in a single degree by assumption, and therefore the map $M(W') \rightarrow Q^{(1)}$ must actually be surjective by Nakayama's lemma. Applying $H_0$, and using the assumptions that $H_0(M(W')) \rightarrow H_0(Q^{(1)})$ is an isomorphism and $H_1(Q^{(1)}) = 0$, we conclude that $H_0(K^{(1)}) = 0$. It follows that $K^{(1)} = 0$, and therefore $Q^{(1)} = M(W')$ is homology acyclic by part three of Proposition \ref{mprop}. The first two terms in (\ref{hom3}) now imply that $H_1(I^{(2)}) = 0$. Proceeding inductively, we eventually reach the conclusion that $H_1(I^{(n)}) = 0$, as desired.\\

We have thus far shown that the second statement is equivalent to the third and fourth, and that the first statement implies these. It only remains to show that the second statement implies the first. Assume that $V$ admits a cofiltration as in the third statement of the theorem, and assume that the factors of this cofiltration are given by the collection $\{M(W_i)\}_{i=1}^g$, with $W_i$ a $k[G_i]$-module. We first observe that $H_0(V)$ is the $\FB_G$-module which is $W_i$ in degree $i$, and zero elsewhere. Indeed, this follows from how the $W_i$ were constructed above. For each $i$, let $\{w_{i,j}\}_{j=1}^{\kappa_1}$ be a generating set for $W_i$. Applying Nakayama's lemma we obtain a surjection
\[
\bigoplus_{i = 1}^g M(i)^{\kappa_i} \twoheadrightarrow V.
\]
Set $V'$ to be the submodule of $V$ generated by lifts of the $w_{i,j}$ with $i < g$. It remains to show that $V/V' = M(W_g)$.\\

Call $Q := V/V'$, and apply $H_0$ to the exact sequence
\[
0 \rightarrow V' \rightarrow V \rightarrow Q \rightarrow 0.
\]
By the previously proven equivalences, we are left with the sequence
\[
0 \rightarrow H_1(Q) \rightarrow H_0(V') \rightarrow H_0(V) \rightarrow H_0(Q) \rightarrow 0
\]
By construction, $H_0(V')$ is the module $H_0(V)$ with the term $W_g$ set to zero, and the map $H_0(V') \rightarrow H_0(V)$ is the obvious inclusion. This implies two things: $H_1(Q) = 0$, and $H_0(Q)$ is the module which is $W_g$ in degree $g$, and zero elsewhere.\\

The structure of $H_0(Q)$ implies that $Q$ is zero up to degree $g$, where it is $W_g$. The identity map on $W_g$ induces a surjection
\[
M(W_g) \twoheadrightarrow Q,
\]
which we claim is an isomorphism. Letting $K$ be the kernel of this map, and using the fact that $H_1(Q) = 0$, we are obtain an exact sequence
\[
0 \rightarrow H_0(K) \rightarrow H_0(M(W_g)) \rightarrow H_0(Q) \rightarrow 0.
\]
The final map is an isomorphism by construction, and so $H_0(K) = 0$. This concludes the proof.\\
\end{proof}

\begin{remark}
As was noted during the proof, we again observe that the filtration constructed above has the property that the $k[G_n]$-modules which appear in the cofactors $M(W)$ are precisely the non-trivial terms of $H_0(V)$.\\
\end{remark}

\begin{remark}
The theorem just proven is the first, and largest, part of Theorem \ref{homacyclic}. One may have noted that very little about the structure of $\FI_G$ specifically was used in the previous proof. Indeed, this theorem will hold for modules over many other categories. Examples of these categories include $\FI_d$, of finite sets with injections decorated by a $d$-coloring of the complement of their image, and VI, of finite vector spaces over a fixed finite field with injective linear maps. The interested reader should see \cite{GL}\cite{L}\cite{SS}\cite{PS} for more on modules over these categories.\\

It is natural for one to ask if we can prove the second half of Theorem \ref{homacyclic} in a more general context. The answer to this question no, and it is most easily illustrated by the following example of Jordan Ellenberg. Let $\Ca$ be the natural numbers, viewed as a poset category. The above theorem will hold in this category. One immediately finds that the $\Ca$-module $M(0)$ is the object which is $k$ in every degree, while $M(1)$ is the object which is 0 in degree 0, and $k$ in all other degrees. In particular, there is a natural embedding $M(1) \hookrightarrow M(0)$, whose cokernel is the object which is $k$ in degree 0, and 0 elsewhere. It is clear that this cokernel is not sharp filtered, and therefore we have a non $\sharp$-filtered object which admits a finite resolution by $\sharp$-filtered objects. It is an interesting question to ask for which categories one has the latter two equivalences of Theorem \ref{homacyclic}.\\
\end{remark}

One technical corollary to Theorem \ref{whomacyclic} is the following.\\

\begin{lemma}\label{needtwo1}
Given an exact sequence
\[
0 \rightarrow V' \rightarrow V \rightarrow V'' \rightarrow 0,
\]
of $\FI_G$-modules which are generated in finite degree such that $V''$ is $\sharp$-filtered, $V'$ is $\sharp$-filtered if and only if $V$ is $\sharp$-filtered.\\
\end{lemma}

\begin{proof}
One applies the zeroth homology functor, and uses the fact that $V''$ is $\sharp$-filtered, to conclude that $H_i(V) = H_i(V')$ for all $i \geq 1$. Theorem \ref{whomacyclic} now implies the lemma.\\
\end{proof}

This fact was first proven in \cite[Proposition A.6]{D} in a much more general context. More recently, it was also proven in \cite[Corollary 3.6]{LY} for $\FI$-modules. In fact, the result proven in these papers is slightly stronger than that given above, as it includes the case where $V'$ and $V$ are known to be $\sharp$-filtered. We will provide a different proof of this strengthening as a consequence of the depth classification theorem in Section \ref{depthclass}.\\

\begin{remark}
The work thus far completed in this paper seems to indicate that $\sharp$-filtered objects are a fundamentally important class in $\FI_G$-mod. One observes that there is a chain of classes
\[
\text{Projective Objects} \subseteq \text{Relatively Projective Objects} \subseteq \text{$\sharp$-Filtered Objects}
\]
In the case where $k$ is a field of characteristic 0, the above inclusions are equalities by Proposition \ref{projclass}. However, in general the inclusions can be proper. For example, if $k$ is a field of characteristic $p > 0$, then an example of Nagpal \cite[Example 3.35]{N}, which was independently discovered by Gan and Li \cite[Section 3]{GL}, shows that there are $\sharp$-filtered objects which are not relatively projective. On the other hand, if $W$ is a non-projective $k[G_n]$-module for some $n$, then $M(W)$ is not projective. It becomes an interesting question to ask whether there is some homological criterion which separates $\sharp$-filtered from relatively projective.\\
\end{remark}

\section{The Church-Ellenberg Approach to Regularity}

\subsection{The Derivative and its Basic Properties}\label{derv}

\begin{definition}\label{dervdef}
Let $V$ be an $\FI_G$-module, and let $\iota$ denote the natural map $\iota:V \rightarrow \So V$. The \textbf{derivative} of $V$ is the $\FI_G$-module
\[
DV = \coker(\iota).
\]
For any $a \geq 0$, we define $D^a$ to be $a$-th iterate of $D$.\\
\end{definition}

Because $\So$ is exact, and because $DV$ is defined as a cokernel, it follows immediately that $D$ is a right exact functor.\\

\begin{definition}
We will follow \cite{CE} and write $H_i^{D^a}$ to denote the $i$-th left derived functor of $D^a$ for any $a \geq 1$.\\
\end{definition}

\begin{proposition}[\cite{CE}, Proposition 3.5 and Lemma 3.6]\label{dervprop}
The derivative functor $D$ enjoys the following properties:
\begin{enumerate}
\item for any $k[G_n]$-module $W$,
\[
D(M(W)) = M(\Res_{G_{n-1}}^{G_n} W);
\]
\item if $V$ is $\sharp$-filtered, then it is acyclic with respect to $D^a$ for all $a \geq 1$;
\item if $V$ is generated in degree $\leq m$, then $DV$ is generated in degree $\leq m-1$. Conversely, if $\deg(D^a V) \leq m$ for some $a,m$, then $V$ is generated in degree $\leq a + m$;
\item for any $\FI_G$-module $V$ there is an exact sequence
\[
0 \rightarrow H_1^D(V) \rightarrow V \stackrel{\iota}\rightarrow SV \rightarrow DV \rightarrow 0;
\]
\item $H_1^D(V) = 0$ if and only if $V$ is torsion free;
\item for any $\FI_G$-module $V$, $H_i^D(V) = 0$ for all $i > 1$;
\item if $\deg(V) \leq n$ then $\deg(DV) \leq n-1$, and $\deg(H_1^D(V)) \leq \deg(V)$.\\
\end{enumerate}
\end{proposition}

\begin{remark}
In the cited paper, the authors only prove that relatively projective objects are acyclic with respect to $D^a$. Part 2 of the previous proposition actually follows immediately from this. Also note that the provided source only proves these statements for $\FI$-modules. The proofs are exactly the same.\\
\end{remark}

One of the main results of the cited paper was to prove that, in the case of $\FI$, the functors $H_i^{D^a}$ all had finite degree. They did this by providing an explicit bound on the degree in terms of certain invariants of $V$ \cite[Theorem 3.8]{CE}. We will eventually be able to do this as well in the case of $\FI_G$-modules.\\

Observe that Propositions \ref{dervprop}, \ref{shiftprop}, and \ref{projclass} imply that both $\So$ and $D$ preserve projective objects. This will allow us to call upon the Grothendieck spectral sequence in the following lemma.\\

\begin{lemma}\label{shiftderv}
There is a natural isomorphism of functors
\[
\So_b \circ D^a \cong D^a \circ \So_b
\]
for all $a,b \geq 1$. More generally, there are natural isomorphisms of functors
\[
\So_b \circ H_i^{D^a} \cong H_i^{D^a} \circ \So_b
\]
for all $b,a \geq 1$ and all $i \geq 0$.\\
\end{lemma}

\begin{proof}
We begin with the first claim. It clearly suffices to show the statement in the case where $a = b = 1$. Let $V$ be an $\FI_G$-module, and let $\tau_n:V_{n+2} \rightarrow V_{n+2}$ denote the isomorphism induced by the transposition $(n+2,n+1)$ paired with the trivial map into $G$. We claim that this map induces an isomorphism $\So D(V)_n \cong D(\So V)_n$.\\

Reviewing how everything is defined, we see that on points
\begin{eqnarray*}
\So D(V)_n &=& V_{n+2}/\im((f^{n+1},\mathbf{1})_\as),\\
D(\So V)_n &=& V_{n+2}/\im((f^n_{+},\mathbf{1})_\as)
\end{eqnarray*}
where $f^{n+1}:[n+1] \rightarrow [n+2]$ is the standard inclusion, $f^n_+:[n+1] \rightarrow [n+2]$ is the map which is the identity on $[n]$ and sends $n+1$ to $n+2$, and where $\mathbf{1}$ is the trivial map into $G$. In particular,
\[
\tau_n(f^{n+1},\mathbf{1})_\as = ((n+2,n+1)f^{n+1},\mathbf{1})_\as = (f^n_{+},\mathbf{1})_\as.
\]
This shows that $\tau_n$ induces an isomorphism between $\So D(V)_n$ and $D( \So V)_n$, as desired. We claim that $\tau$ is actually a map of $\FI_G$-modules. Let $(f,h):[n] \rightarrow [m]$ be an map in $\FI_G$. Then the map induced in $D(\So V)$ will be the image under the quotient of $(f_{++},h_{++})_\as:V_{n+2} \rightarrow V_{m+2}$, where $f_{++}$ agrees with $f$ on $[n]$ and sends $n+i$ to $m+i$ for $1 \leq i \leq 2$, and $h_{++}$ agrees with $h$ on $[n]$ and sends both $n+1$ and $n+2$ to 1. On the other hand, the map induced on $\So D(V)$ will also be the image in the quotient of $(f_{++},h_{++})_\as$. Then,
\[
(f_{++},h_{++})_\as \tau_n = (f_{++}(n+1,n+2), h_{++})_\as = ((m+1,m+2)f_{++},h_{++})_\as = \tau_m(f_{++},h_{++})_\as.
\]
The fact that the collection of $\tau$ gives us a map of functors is easily checked.\\

The second statement is largely homological formalism. Let $H_i^{D^a \circ \So_b}$ denote the $i$-th left derived functor of $D^a \circ \So_b$, and similarly define $H_i^{\So_b \circ D^a}$. Because $\So_b$ is exact, the Grothendieck spectral sequences for both of these derived functors have a single row, or column, respectively. In particular,
\[
H_i^{D^a} \circ \So_b = H_i^{D^a \circ \So_b} = H_i^{\So_b \circ D^a} = \So_b \circ H_i^{D^a}.
\]
This concludes the proof.\\
\end{proof}

We have already discussed the fact that the modules $H_i^{D^a}(V)$ have finite degree for all $i$ and $a$ whenever $V$ is finitely generated. The previous lemma implies $\So_bV$ will be acyclic with respect to all derivative functors for $b$ sufficiently large. We will reinterpret this later in terms of depth.\\

Before we finish this section, we take a moment to point out the exact sequence relating the derived functors of varying derivatives. In particular, if one writes $D^a = D \circ D^{a-1}$, then $H_p^{D^a}$ can be computed using the Grothendieck spectral sequence. By part 6 of Proposition \ref{dervprop}, we know that this spectral sequence only has two columns. Thus,

\begin{eqnarray}
0 \rightarrow DH_{i}^{D^a}(V) \rightarrow H_i^{D^{a+1}}(V) \rightarrow H_1^D(H_{i-1}^{D^{a}}(V)) \rightarrow 0. \label{groth}
\end{eqnarray}
\text{}\\

\subsection{The Relationship Between the Derivative and Regularity}

\begin{definition} \label{dervreg}
Let $V$ be an $\FI_G$-module. We define the \textbf{derived regularity} of $V$ to be the quantity $\dreg(V) := \sup\{\deg(H_1^{D^{a}}(V))\}_{a=1}^\infty \in \N \cup \{-\infty, \infty\}$. If $V$ is acyclic with respect to $D^a$ for all $a$, then we set $\dreg(V) = -\infty$.\\

We also define the \textbf{derived width} of $V$ to be the quantity $\dwidth(V) = \sup\{\deg(H_1^{D^a}(V)) + a\}_{a=1}^\infty \in \N \cup \{-\infty,\infty\}$.\\
\end{definition}

We will find the derived width and regularity of a module to be of great importance in what follows. Our first goal will be to show that both of these quantities are bounded relative to one another.\\

\begin{proposition}\label{finwidth}
Let $V$ be an $\FI_G$-module which is presented in finite degree and generated in degree $\leq d$. Then,
\[
\dreg(V)+1 \leq \dwidth(V) \leq \dreg(V) + \max\{hd_1(V),d\}
\]
\end{proposition}

\begin{proof}
The proposition is clear if $\dreg(V) = \infty$, so we assume not. By assumption there is a presentation
\[
0 \rightarrow K \rightarrow M \rightarrow V \rightarrow 0
\]
such that $M$ is generated in degree $\leq d$, and $K$ is generated in degree $\leq \max\{hd_1(V),d\}$ (See Remark \ref{reldegdef}). Applying $D^a$ to the above sequence, with $a > \max\{hd_1(V),d\}$, and applying the second and third part of Proposition \ref{dervprop}, it follows that $H_1^{D^a}(V) = 0$. Therefore, any $a$ for which $H_1^{D^a}(V) \neq 0$ must satisfy $H_1^{D^a}(V) + a \leq \dreg(V) + \max\{hd_1(V),d\}$.\\

It follows from the above work that we may find some $a$ such that $\deg(H_1^{D^a}(V)) = \dreg(V)$. Then,
\[
\dreg(V) +1= \deg(H_1^{D^a}(V)) + 1\leq \deg(H_1^{D^a}(V)) + a \leq \dwidth(V).
\]
This completes the proof.\\
\end{proof}

For the purposes of relating the derived regularity to the usual regularity of a module, we will need to know how the degrees of $H_i^{D^a}(V)$ depend on $i$ and $a$.\\

\begin{proposition}\label{regbound}
Let $V$ be an $\FI_G$-module which is presented in finite degree. Then, for all $a$ and $i \geq 1$,
\[
\deg(H_i^{D^a}(V)) \leq \dwidth(V)-1 + i - a,
\]
\end{proposition}

\begin{proof}
This proposition was essentially proven in \cite[Theorem 3.8]{CE} for $\FI$-modules. We follow their general strategy here.\\

The exact sequence (\ref{groth}) implies that the first $a$ for which $V$ is not acyclic with respect to $D^a$ must have $H_1^{D^a}(V)  \neq 0$. If it is the case that $\dreg(V) = -\infty$, this implies that $V$ is acyclic with respect to $D^a$ for all $a$, and the above inequality holds. We may therefore assume that $\dreg(V)$ is finite. We proceed by induction on $a$. If $a = 1$, then Proposition \ref{dervprop} implies that $H_i^D(V) = 0$ for all $i > 1$, and the bound holds trivially in this case. If $i = 1$, then the bound follows from the definition of derived width.\\

Assume that the statements hold up to some $a$. We first note that the bound holds when $i = 1$ by the definition of derived width. We once again call upon (\ref{groth}) and write,
\[
0 \rightarrow DH_{i}^{D^a}(V) \rightarrow H_i^{D^{a+1}}(V) \rightarrow H_1^D(H_{i-1}^{D^{a}}(V)) \rightarrow 0.
\]
Assuming that $i > 1$, we know that
\begin{eqnarray*}
\deg(H_{i}^{D^a}(V)) &\leq& \dwidth(V)-1 + i - a\\
\deg(H_{i-1}^{D^{a}}(V)) &\leq& \dwidth(V)-1 + i - a -1.
\end{eqnarray*}
Applying part 7 of Proposition \ref{dervprop} to both of these inequalities, we find that
\begin{eqnarray*}
\deg(DH_{i}^{D^a}(V)) &\leq& \dwidth(V)-1 + i - a -1\\
\deg(H_1^D(H_{i-1}^{D^{a}}(V))) &\leq& \dwidth(V)-1 + i - a -1.
\end{eqnarray*}
This proves the claim.\\
\end{proof}

We finish this section by showing the connection between derived regularity, and the previously mentioned notion of regularity .\\

\begin{proposition}\label{homreg}
Let $V$ be an $\FI_G$-module which is presented in finite degree and has  $\leq r$. Then,
\[
hd_i(V) \leq \max\{\dwidth(V)-1,\max\{hd_1(V),d\}-1\} + i,
\]
for all $i \geq 1$. In particular,
\[
hd_i(V) \leq \dreg(V) + \max\{hd_1(V),d\} - 1.
\]\\
\end{proposition}

\begin{proof}
As with the previous proposition, this statement largely follows from the work in \cite[Theorem 3.9]{CE}. For the remainder of the proof, we fix $N := \max\{\dwidth(V)-1,\max\{hd_1(V),d\}-1\}$. We begin with a projective resolution of $V$,
\[
\ldots \rightarrow M_1 \rightarrow M_0 \rightarrow V \rightarrow 0.
\]
Writing this as a collection of short exact sequences, we have
\[
0 \rightarrow X_{i+1} \rightarrow M_i \rightarrow X_{i} \rightarrow 0
\]
for some modules $X_i$, with $X_0 = V$. By repeatedly applying Nakayama's lemma, we may assume without loss of generality that $\deg(H_0(M_i)) = \deg(H_0(X_i))$ for all $i$ Our first objective will be to prove the following claim:
\[
\deg(H_0(X_i)) \leq N + i
\]
for all $i \geq 1$. We prove this by induction on $i$. If $i = 1$, then
\[
\deg(H_0(X_1)) \leq \max\{hd_1(V),d\} + 1 \leq N + 1.
\]
Assume that $\deg(H_0(X_i)) \leq N + i$ for some $i \geq 1$, and consider the exact sequence
\begin{eqnarray}
0 \rightarrow X_{i+1} \rightarrow M_i \rightarrow X_{i} \rightarrow 0.\label{exact1}
\end{eqnarray}
By the assumption made at the beginning of the proof, we know that $M_i$ is generated in degree $\leq N + i$. Part 3 of Proposition \ref{dervprop} implies that $D^{N+i+1}M_i = 0$. Applying this functor to (\ref{exact1}), we therefore find
\[
0 \rightarrow H_1^{D^{N+i+1}}(X_i) \rightarrow D^{N+i+1}X_{i+1} \rightarrow 0,
\]
where we have used the fact that $M_i$ is projective to imply the leading zero. We know that $H_1^{D^{N+i+1}}(X_i) = H_{i+1}^{D^{N+i+1}}(V)$, and Proposition \ref{regbound} implies
\[
0 \geq \dwidth(V)-1 + i+1 - N - i - 1 \geq \deg(H_{i+1}^{D^{N+i+1}}(V)) = \deg(H_1^{D^{N+i+1}}(X_i)) = \deg(D^{N+i+1}X_{i+1}).
\]
The third part of Proposition \ref{dervprop} implies that $X_{i+1}$ is generated in degree $\leq N+i+1$, as desired. To finish the proof, one applies the zeroth homology functor to (\ref{exact1}) to obtain,
\[
0 \rightarrow H_1(X_{i}) \rightarrow H_0(X_{i+1}).
\]
This shows that $\deg(H_1(X_i)) \leq N + i + 1$. We know that $H_1(X_i) = H_{i+1}(V)$, finishing the proof of the proposition.\\
\end{proof}

\begin{remark}
One should note the following consequence of the above proof. If $V$ is an $\FI_G$-module which is presented in finite degree, and which has finite derived width, then $V$ admits a projective resolution whose every member is generated in finite degree.\\
\end{remark}

One application of this result is in bounding the regularity of $\FI_G$-modules with finite degree.\\

\begin{corollary}\label{finlen}
Let $V$ be an $\FI_G$-module with $\deg(V) < \infty$. Then,
\[
hd_i(V) \leq i + \deg(V) 
\]
\end{corollary}

\begin{proof}
To begin, we claim that $\dwidth(V)-1 \leq \deg(V)$. In particular, for all $a \geq 1$
\begin{eqnarray}
\deg(H_1^{D^a}(V)) \leq \deg(V) + 1 - a. \label{finlenderv}
\end{eqnarray}
We prove this claim by induction on $a$. If $a = 1$, then the bound holds by the fourth part of Proposition \ref{dervprop}. Assume we have proven the bound for some $a \geq 1$. The sequence (\ref{groth}) implies
\[
0 \rightarrow DH_1^{D^a}(V) \rightarrow H_1^{D^{a+1}}(V) \rightarrow H_1^D(D^aV) \rightarrow 0.
\]
The module $D^aV$ has degree at most $\deg(V)-a$, and therefore $H_1^D(D^aV)$ also has degree at most $\deg(V)-a$ by the last part of Proposition \ref{dervprop}. On the other hand, induction tells us that $\deg(H_1^{D^a}(V)) \leq \deg(V) - a + 1$, and therefore $\deg(DH_1^{D^a}(V)) \leq \deg(V)-a$ by the third part of Proposition \ref{dervprop}. The above exact sequence implies the claim.\\

To finish the proof, one simply notes that $hd_1(V) \leq \deg(V)$, and applies Proposition \ref{homreg}.\\
\end{proof}

This same bound is found in \cite[Theorem 1.5]{L} using different methods.\\

\subsection{Bounding the Derived Width and the Proof of Theorem \ref{finitereg}}

The purpose of this section is to generalize the methods of Church and Ellenberg to provide explicit bounds on the regularity in terms of the generating and relation degrees of the module $V$. More specifically, we will provide bounds on the derived width of $V$ and apply Proposition \ref{homreg}.\\

The notation used in this section traces its origins to \cite{CE}. Indeed, one may consider this section as an expository account of \cite[Section 2]{CE}, where we are careful in generalizing relevant definitions to $\FI_G$-modules.\\

\begin{definition}
Let $V$ be an $\FI_G$-module, and fix a pair of integers $i \leq n$. Then we write $V_{n - \{i\}}$ to denote the submodule of $V_n$ generated by images of induced maps of the form $(f_i,\mathbf{1})_\as$, where $\mathbf{1}:[n-1] \rightarrow G$ is the trivial map and $f_i:[n-1] \rightarrow [n]$ is the map
\[
f_i(x) = \begin{cases} x &\text{ if $x < i$}\\ x+1 &\text{ otherwise.}\end{cases}.
\]
\end{definition}

\begin{remark}
Because it will be important later, we note that, in fact, $V_{n-\{i\}}$ is the submodule of $V_n$ generated by images of maps of the form $(f_i,g)$, where $g:[n-1] \rightarrow G$ is any map. Indeed, this follows from the identity
\[
(f_i,g)_\as(v) = (f_i,\mathbf{1})_\as ((id,g)_\as(v)).
\]
\text{}\\
\end{remark}

With this notation it is immediate that for any $a \geq 1$,
\[
D^aV_n = V_{n+a} / \sum_{i=1}^a V_{n+a - \{n+i\}}.
\]

Let $0 \rightarrow K \rightarrow M \rightarrow V \rightarrow 0$ be a presentation for $V$. Then applying $D^a$ it follows that
\[
H_1^{D^a}(V) = \ker(K_{n+a} / \sum_{i=1}^a K_{n+a - \{n+i\}} \rightarrow M_{n+a} / \sum_{i=1}^a M_{n+a - \{n+i\}}).
\]
This implies the following proposition.\\

\begin{proposition}\label{ecombi}
The degree of $H_1^{D^a}(V)$ is the smallest integer $m$ such that
\[
K_{n+a} \cap \sum_{i=1}^a M_{n+a - \{n+i\}} = \sum_{i=1}^a K_{n+a - \{n+i\}}
\]
for all $n > m$.\\
\end{proposition}

This reformulation was first observed by Church and Ellenberg in \cite{CE}. In that paper, it is shown that bounding the value $m$ in the above proposition is deeply rooted in the combinatorics of $\FI$-modules. For the remainder of this section, we show that the techniques of \cite{CE} can be applied to $\FI_G$.\\

\begin{definition}
Fix a non-negative integer $n$, and let $i \neq j$ be elements of $[n]$. Then we set 
\[
J_i^j = (id,\mathbf{1}) - ((i,j),\mathbf{1}) \in \Z[G_n],
\]
where $\mathbf{1}$ is the trivial map into $G$. For any non-negative integer $m$, we define $I_m$ to be the ideal of $\Z[G_n]$ generated products of the form
\[
J_{i_1}^{j_1}\cdots J_{i_m}^{j_m}
\]
where all of the indices are distinct elements of $[n]$.\\

For any non-negative integer $b$, we write $\Sigma(b)$ to denote the collection of $b$-element subsets $S \subseteq [2b]$ such that the $i$-th smallest element of $S$ is at most $2i-1$. For any $1 \leq a \leq b$ we write $\Sigma(a,b) = \{S \in \Sigma(b) \mid [a] \subseteq S\}$.\\

For any $S \in \Sigma(b)$, we may write its entries in increasing order as $s_1 < \ldots < s_b$, and we may write the entries of its compliment in increasing order as $t_1 <\ldots < t_b$. Then define
\[
J_S := \prod_i J_{s_i}^{t_i}.
\]
\end{definition}

The paper \cite{CE} spends some time discussing the interesting combinatorial properties of $\Sigma(a,b)$, and its connection with the Catalan numbers. One observation they make is the following.\\

\begin{lemma}[\cite{CE}, Section 2.1]\label{combi}
Let $S \subseteq [2b]$ have $b$ elements, and write its elements in increasing order as $s_1 < \ldots < s_b$. Assume that $U \subseteq [n]$ is such that $S \subseteq U$. For any $b$ distinct elements $i_1 < \ldots < i_b$ in $[n] - U$, write $H$ for the subgroup of $\Sn_n$ generated by the disjoint transpositions $(i_p,s_p)$. Then,
\[
S \in \Sigma(b) \implies \text{ $U$ is lexicographically first among $\{\sigma \cdot U \mid \sigma \in H\}$.}
\]
\text{}\\
\end{lemma}

We note that the in the case of $\FI$, defining $J_i^j$ does not require a choice of map $[n] \rightarrow G$. In the above definition we have chosen the trivial map for the following reason. If $(f,g):[m] \rightarrow [n]$ is any map in $\FI_G$, and $(\sigma,\mathbf{1}) \in G_n$, then $(\sigma,\mathbf{1})\circ(f,g) = (\sigma \circ f, g)$. In other words, choosing the map into $G$ to be trivial grants us the ability to often times ignore the map $g$. This choice will allow us to use the arguments of \cite{CE}.\\

For the remainder of this section we fix integers $r \leq n$, and write $F := \Z[\Hom_{\FI_G}([r],[n])]$.\\

\begin{definition}
Let $1 \leq a \leq b$ be integers, and assume $S \in \Sigma(b)$. Then we define the following submodules of $F$,
\begin{enumerate}
\item $F^{\neq S} := ((f,g) \in F \mid S \not\subseteq \im f)$;
\item $F^b := ((f,g) \in F \mid \forall S \in \Sigma(b), S \not\subseteq \im f) = \cap_{S \in \Sigma(b)}F^{\neq S}$;
\item $F^{a,b} := ((f,g) \in F \mid \forall S \in \Sigma(a,b), S \not\subseteq \im f) = \cap_{S \in \Sigma(a,b)}F^{\neq S}$;
\item $F_{=S} := ((f,g) \in F \mid [2b] \cap \im f = S)$.\\
\end{enumerate}
\end{definition}

\begin{proposition}[\cite{CE} Propositions 2.3, 2.4, 2.6, and Lemma 2.5]\label{combiprop}
Let $a,b,m,$ and $p$ be non-negative integers.
\begin{enumerate}
\item If $n \geq b+r$, then
\[
F = I_b\cdot F + F^b.
\]
\item If $a \leq b$ and $2b \leq n$, then
\[
F^{a,b+1} \subseteq F^{a,b} + \sum_{S \in \Sigma(a,b)} J_S \cdot F_{=S}.
\]
\item If $V$ is an $\FI_G$-module which is generated in degree $\leq r$, then $I_{r+1}\cdot M_n = 0$ for all $n \geq 0$.
\item Given $(f,g):[r] \rightarrow [n]$ and $\{i_1,j_1,\ldots,i_m,j_m\} \subseteq [n]$, if $\im f \cap \{i_p,j_p\} = \emptyset$ for some $p$, then $J_{i_1}^{j_1}\cdots J_{i_m}^{j_m} (f,g) = 0$.\\
\end{enumerate}
\end{proposition}

\begin{proof}
The cited source proves all of these claims for $\FI$-modules. As stated previously, by our choice in defining the elements $J_i^j$ all of these arguments will work almost verbatim. We work through the proof of the first statement as an example of this.\\

It is clear that $F = I_b\cdot F + F^b$ if and only if the latter group contains all the canonical basis vectors. Assume that this is not the case, and pick $(f,g) \notin I_b\cdot F + F^b$ so that the image of $f$ is lexicographically largest among all basis vectors missing from $I_b\cdot F + F^b$. Note that the image of $f$ has size $r$, and there are therefore at least $b$ elements in the compliment of this image, by assumption. Write $i_1 < i_2 < \ldots < i_b$ for some sequence of elements in $[n] - \im f$. By assumption, $(f,g) \notin F^b$, which implies that there is some $S \in \Sigma(b)$ such that $S$ is contained in the image of $f$. Write the elements of $S$ in increasing order as $s_1 < \ldots < s_b$, and consider $((s_p,i_p),\mathbf{1})\cdot (f,g) = ((s_p,i_p)f,g)$. Applying Lemma \ref{combi} with $U$ being the image of $f$, we conclude that the image of $(s_p,i_p)f$ is lexicographically larger than the image of $f$. By our assumption on $f$, it must be the case that $((s_p,i_p)f,g) \in I_b\cdot F + F^b$. Calling $J := J_{i_1}^{s_1}\cdots J_{i_p}^{s_p}$, it follows from definition that
\[
J-(id,\mathbf{1}) = \sum_{id \neq \sigma \in H} (-1)^{\sigma}(\sigma,\mathbf{1})
\]
where $H$ is the subgroup of $\Sn_n$ generated by the transpositions $(s_p,i_p)$. In particular, $(J-(id,\mathbf{1}))(f,g) \in I_b\cdot F + F^b$ by the previous computation. On the other hand, $J\cdot (f,g) \in I_b \cdot F \subseteq I_b\cdot F + F^b$ by definition. This shows that $(f,g) \in I_b\cdot F + F^b$, which is a contradiction.\\
\end{proof}

We are now ready to state and prove the main theorem of this section. As with all the previous statements, the proof of the following proceeds in precisely the same way it did in the case of $\FI$-modules. \\

\begin{theorem}[\cite{CE}, Theorem A]\label{boundwidth}
Let $K \subseteq M$ be torsion free $\FI_G$-modules, and assume that $M$ is generated in degree $\leq d$ and $K$ is generated in degree $\leq r$. Then for all $n \geq \min\{r,d\} + r+1$, and all $a \leq n$,
\[
K_n \cap \sum_{i=1}^a M_{n-\{i\}} = \sum_{i=1}^a K_{n-\{i\}}.
\]
In particular, if $V$ is an $\FI_G$-module which is generated in degree $\leq d$ and related in degree $\leq r$ then $\dwidth(V) \leq \min\{r,d\} + r$.\\
\end{theorem}

\begin{proof}
We first note that the second statement follows as an immediate consequence of the first statement, Proposition \ref{ecombi}, and the definition of derived regularity. It therefore suffices to prove the first statement. Due to its similarity with the proof in the provided source, we only give an outline here for the convenience of the reader.\\

Our first reduction will be to assume that $K$ and $M$ are $\FI_G$-modules over $\Z$. Observe that if $V$ is an $\FI_G$-module over a ring $k$, then it can also be considered as an $\FI_G$-module over $\Z$. It is clear that doing this does not change whether $V$ is generated in finite degree.\\

Because $K$ is generated in degree $\leq r$, the map $F \otimes K_r \rightarrow K_n$ is surjective for all $n > r$. Let $a$ be as in the statement of the theorem, and let $b \geq a$ be an integer. We define the following submodules of $K_n$
\[
K^b := \im(F^b \otimes K_r \rightarrow K_n), \text{ } K^{a,b} := \im(F^{a,b} \otimes K_r \rightarrow K_n)
\]
The remainder of the proof proceeds in the following steps. One first shows that $K^{a, \min\{r,d\}+r+1} = K_n$, and then that $K^{a,b+1} \cap \sum_i M_{n-\{i\}} \subseteq K^{a,b}$. At this point induction on $b$, beginning at $\min\{r,d\}+r+1$ and ending at $a$, implies that $K_n \cap \sum_i M_{n-\{i\}} \subseteq K^{a,a}$. Note that $K^{a,a}$ is, by definition, the submodule of $K_n$ generated by images of induced maps $(f,g)$ such that $[a]$ is not contained in the image of $f$. In other words, using the remark at the beginning of the section,
\[
K^{a,a} = \sum_{i=1}^a K_{n - \{i\}}.
\] 
This proves the theorem.\\

Applying the first part of Proposition \ref{combiprop}, we find
\[
K_n = \im(F \otimes K_r) = \im(I_{\min\{r,d\}+r+1}F + F^{\min\{r,d\}+r+1} \otimes K_r) = I_{\min\{r,d\}+r+1} \cdot K_n + K^{\min\{r,d\}+r+1}
\]
Applying the third part of Proposition \ref{combiprop}, we have that $I_{\min\{r,d\}+r+1} \cdot K_n = 0$, and therefore
\[
K_n = K^{\min\{r,d\}+r+1} \subseteq K^{a,{\min\{r,d\}+r+1}} \subseteq K_n.
\]

The second claim - that $K^{a,b+1} \cap \sum_i M_{n-\{i\}} \subseteq K^{a,b}$ - is considerably more subtle. We direct the reader to the original source for the details.\\
\end{proof}

\begin{corollary}\label{betterboundwidth}
Let $V$ be an $\FI_G$-module which is presented in finite degree, and which is generated in degree $\leq d$. Then
\[
\dwidth(V) \leq hd_1(V) + \min\{d,hd_1(V)\}.
\]
\end{corollary}

\begin{proof}
This follows from the techniques discussed in Remark \ref{reldegdef}. One simply notes that the functors $D^a$ are right exact, and that $\sharp$-filtered objects are acyclic with respect to these functors.\\
\end{proof}

Using all we have thus far learned, we can finally prove Theorem \ref{finitereg}.\\

\begin{proof}[Proof of Theorem \ref{finitereg}]
One applies Corollary \ref{betterboundwidth}, and Proposition \ref{homreg}.\\
\end{proof}

\section{Depth}

\subsection{Definition and the Classification Theorem} \label{depthclass}

The ``usual'' first definition of depth in classical commutative algebra is stated in terms of regular sequences (see \cite[Chapter 18]{E} for the classical theory). One then proves a relationship between this definition and the Koszul complex. It is not immediately obvious what one would mean by a regular sequence in an $\FI_G$-module, however. Another approach one might consider is defining depth through some kind of local cohomology theory. This may not work in this setting, as $\FI_G$-mod over a field of characteristic $p$, for example, will not have sufficiently many injectives. It is therefore not even clear that a local cohomology theory exists in this case. Note that if $k$ is a field of characteristic 0, then Sam and Snowden have developed a theory of local cohomology and depth for $\FI$-modules \cite{SS3}. The definition we give now seems completely divorced from the classical theory. We hope that through the proofs that follow, one can develop a better idea of why this is the right definition. In Section \ref{classical}, we explore a more classically motivated definition, and prove that it is equivalent.\\

\begin{definition}
Let $V$ be an $\FI_G$-module. Then we define the \textbf{depth} $\de(V)$ of $V$ to be the infimum
\[
\de(V) := \inf\{a \mid H_1^{D^{a+1}}(V) \neq 0\} \in \N \cup \{\infty\}\},
\]
where we use the convention that the infimum of the empty set is $\infty$.\\
\end{definition}

\begin{lemma}\label{welldef}
Let $V$ be an $\FI_G$-module which is presented in finite degree, and for which there is some $\delta \geq 0$ such that $V$ is acyclic with respect to $D^a$ for all $a \leq \delta$, while $V$ is not acyclic with respect to $D^{\delta+1}$. Then for all $l \geq 1$, $H_{l}^{D^{\delta+l}}(V) \neq 0$, while $H_{i}^{D^{\delta+l}}(V) = 0$ for $i > l$.
\end{lemma}

\begin{proof}
We proceed by induction on $l$. If $l = 1$. We know that $H_i^{D^{\delta+1}}(V) \neq 0$ for some $i$. On the other hand, if one plugs in any $i > 1$ and $a = \delta+1$ into (\ref{groth}), then one finds $H_i^{D^{\delta+1}}(V) = 0$. This shows that $H_1^{D^{\delta+1}}(V) \neq 0$, while $H_i^{D^{\delta+1}}(V) = 0$ for all $i > 1$.\\

Assume that $H_{l}^{D^{\delta+l}}(V) \neq 0$. Then the sequence (\ref{groth}) shows
\[
0 \rightarrow DH_{l+1}^{D^{\delta+l}}(V) \rightarrow H_{l+1}^{D^{\delta+l+1}}(V) \rightarrow H_1^D(H_{l}^{D^{\delta+l}}(V)) \rightarrow 0. 
\]
By assumption $H_{l}^{D^{\delta+l}}(V) \neq 0$. Proposition \ref{regbound} and Theorem \ref{boundwidth} tell us that $H_{l}^{D^{\delta + l}}(V)$ has finite degree, and therefore has nontrivial torsion. The fifth part of Proposition \ref{dervprop} implies that $H_1^D(H_{i}^{D^{\delta}}(V)) \neq 0$, whence $H_{l+1}^{D^{\delta+l+1}}(V) \neq 0$. If $i > l$, then the above sequence, and induction, implies our desired vanishing.\\
\end{proof}

Lemma \ref{welldef} implies the following alternative characterization of depth.\\

\begin{proposition}
Let $V$ be presented in finite degrees. Then,
\[
\de(V) =\sup\{a \mid V \text{ is acyclic with respect to $D^a$}\}.
\]
\end{proposition}

Our next goal will be to show the following incremental property of depth. It shows how the depth of a module relates to the depth of its syzygies.\\

\begin{lemma}\label{adepth}
Given an exact sequence of $\FI_G$-modules which are presented in finite degree,
\[
0 \rightarrow V' \rightarrow X \rightarrow V \rightarrow 0,
\]
such that $X$ is $\sharp$-filtered, $\de(V') = \de(V) + 1$.\\
\end{lemma}

\begin{proof}
We begin by recalling that $X$ is acyclic with respect to $D^a$ for all $a \geq 1$ (Proposition \ref{dervprop}). Applying $D^a$ to the given exact sequence, it therefore follows that $H_i^{D^a}(V') = H_{i+1}^{D^a}(V)$ for all $i,a \geq 1$. If $\de(V) = \infty$, then this immediately implies the same about $\de(V')$. If $\de(V) = \delta < \infty$, then these equalities imply that $H_i^{D^a}(V') = 0$ at least up to $a = \delta$. Moreover, Lemma \ref{welldef} tells us that $H_1^{D^{\delta+1}}(V') = H_2^{D^{\delta+1}}(V) = 0$, while $H_1^{D^{\delta+2}}(V') = H_{2}^{D^{\delta+2}}(V) \neq 0$, showing that $\de(V') = \delta+1$.\\
\end{proof}

To prove the classification theorem, we will need a collection of technical lemmas. The following lemma was originally proven by Nagpal in \cite[Lemma 2.2]{N} for $\FI$-modules, and was later reproven by Li and Yu in \cite[Lemma 3.3]{LY}.\\

\begin{lemma}\label{nagpal}
Let $W$ be a $k[G_m]$-module for some $m$, and assume that there is an exact sequence
\[
0 \rightarrow U \rightarrow M(W) \rightarrow V \rightarrow 0,
\]
such that $U$ is generated in degree $\leq m$. Then $U$ and $V$ are both relatively projective.\\
\end{lemma}

\begin{proof}
Applying $H_0$ to the given exact sequence, we find that
\[
0 \rightarrow H_1(V) \rightarrow H_0(U) \rightarrow H_0(M(W)) \rightarrow H_0(V) \rightarrow 0
\]
Because $U$ was generated in degree $m$, it follows that $H_0(U)$ is only supported in degree $m$. Because $U$ is a submodule of $M(W)$, it follows that $H_0(U)_m = U_m$, and therefore the map $H_0(U) \rightarrow H_0(M(W))_m$ is injective. We conclude that $H_1(V) = 0$. Our result now follows by Theorem \ref{whomacyclic}.\\
\end{proof}

The following lemma was proven in a slightly stronger form in \cite[Lemma 3.11]{LY} using different methods.\\

\begin{lemma}\label{techlem}
Let $V$ be an $\FI_G$-module which is generated in degree $\leq m$, and assume that $H_1^{D^{m+1}}(V) = 0$. Then the relation degree of $V$ is $\leq m$. In particular, if $V$ is generated in finite degree and has infinite depth, then $V$ is presented in finite degree.\\
\end{lemma}

\begin{proof}
By assumption there is an exact sequence,
\[
0 \rightarrow K \rightarrow M \rightarrow V \rightarrow 0
\]
where $M$ is free and generated in degree $\leq m$. Applying the functor $D^{m+1}$ to this sequence, and using Proposition \ref{dervprop} as well as our assumption, we find that $D^{m+1}K = 0$. Proposition \ref{dervprop}  implies that $K$ is generated in degree $\leq m$.\\
\end{proof}

\begin{theorem}[The Depth Classification Theorem]
Let $V$ be an $\FI_G$-module which is presented in finite degree. Then:
\begin{enumerate}
\item $\de(V) = 0$ if and only if $V$ has torsion;
\item $\de(V) = \infty$ if and only if $V$ is $\sharp$-filtered;
\item $\de(V) = \delta$ is a positive integer if and only if there is an exact sequence
\[
0 \rightarrow V \rightarrow X_{\delta-1} \rightarrow \ldots \rightarrow X_{0} \rightarrow V' \rightarrow 0
\]
where $X_i$ is sharp filtered for all $i$, $hd_0(V) = hd_0(X_{\delta-1})$, $hd_0(X_i) > hd_0(X_{i-1}) > hd_0(V')$ for all $i$, and $V'$ is some $\FI_G$-module with torsion.\\
\end{enumerate}
\end{theorem}

\begin{proof}
The first statement is just a rephrasing of the fifth part of Proposition \ref{dervprop}.\\

The backwards direction of the second statement is a rephrasing of part 2 of Proposition \ref{dervprop}. We will prove the forward direction by induction on the generating degree. If $V = 0$, then it is $\sharp$-filtered by definition. Assume that $V$ is generated in non-negative degree $\leq m$, and let $V'$ be the submodule of $V$ generated by the elements of $V_i$ with $i < m$. Then we have an exact sequence,
\[
0 \rightarrow V' \rightarrow V \rightarrow Q \rightarrow 0
\]
where $Q$ is generated in degree $m$. Applying the $H_0$ functor we find that
\[
H_1(V) \rightarrow H_1(Q) \rightarrow H_0(V')
\]
By assumption, $\deg(H_0(V')) \leq m-1$, while $\deg(H_1(V))$ is at most the relation degree of $V$, which is at most $m$ by Lemma \ref{techlem}. It follows that $\deg(H_1(Q)) \leq m$.\\

If,
\[
0 \rightarrow K \rightarrow M(W) \rightarrow Q \rightarrow 0
\]
is an exact sequence, with $W$ a $k[G_m]$-module, then applying $H_0$ we find
\[
0 \rightarrow H_1(Q) \rightarrow H_0(K) \rightarrow H_0(M(W))
\]
It follows that
\[
hd_0(K) \leq \max\{hd_0(M(W)),hd_1(Q)\} \leq m.
\]
Nakayama's lemma and Lemma \ref{nagpal} imply that $Q$ is $\sharp$-filtered. In particular, it has infinite depth and therefore $V'$ also has infinite depth. Applying Lemma \ref{techlem} we know that $V'$ is presented in finite degree, and therefore induction tells us that it is also $\sharp$-filtered. Lemma \ref{needtwo1} implies that $V$ is $\sharp$-filtered, as desired.\\

The backwards direction of the third statement follows from Lemma \ref{adepth}. Conversely, assume that $V$ is presented in finite degree, and $0 < \de(V) = \delta < \infty$. Theorems \ref{boundwidth} and \ref{shiftreg} (proven in the next section) imply that we may find some $b \gg 0$ such that $\So_bV$ is $\sharp$-filtered. Because $V$ has positive depth, it is torsion free, and therefore the map $V \rightarrow \So_bV$ is an embedding. The cokernel of this map is also presented in finite degree by Lemma \ref{inddeg}, and its relation and generating degrees are strictly smaller. The claim now follows from an application of Lemma \ref{adepth} and induction on the generating degree.\\
\end{proof}

\begin{remark}
Lemma \ref{techlem} guarantees that any infinite depth module which is generated in finite degree must actually be related in finite degree as well. The second equivalence of the classification theorem may therefore be stated under the assumption that $V$ is generated in finite degree.\\

The proof that infinite depth modules are equivalent to $\sharp$-filtered modules was significantly different in the earlier versions of this paper. In these versions, the techniques used in the proof only worked if $k$ was a Noetherian ring, $G$ was finite, and $V$ was finitely generated. The recent preprint \cite{LY} proves various new properties of the derivative for $\FI$-modules. The proof of Theorem \ref{depthclass} given here has adapted certain techniques of the paper \cite{LY}, which has allowed us to generalize our original result.\\
\end{remark}

One immediate consequence of Theorem \ref{depthclass} is a bound on $\de(V)$ in terms of the generating degree of $V$.\\

\begin{corollary}
Let $V$ be an $\FI_G$-module which is presented in finite degree, and is of finite depth. Then $\de(V) \leq hd_0(V)$.\\
\end{corollary}

As promised earlier, the classification theorem implies the following strengthening of Lemma \ref{needtwo1}.\\

\begin{corollary}\label{needtwo}
Let
\[
0 \rightarrow V' \rightarrow V \rightarrow V'' \rightarrow 0
\]
be an exact sequence of $\FI_G$-modules which are generated in finite degree. Then any two of $V', V,$ or $V''$ are $\sharp$-filtered if and only if the third is as well.\\
\end{corollary}

\begin{proof}
Lemma \ref{needtwo1} deals with all cases except that in which $V'$ and $V$ are $\sharp$-filtered. In this case Lemma \ref{adepth} implies that $V''$ must have infinite depth, and the classification theorem then implies that $V''$ is $\sharp$-filtered, as desired.\\
\end{proof}

This lemma is the last piece needed to prove Theorem \ref{homacyclic}.\\

\begin{proof}[Proof of Theorem \ref{homacyclic}]
The equivalence of the first four statements was proven in Theorem \ref{whomacyclic}. It is clear that any of these four implies the fifth statement. Conversely, if $V$ admits a finite resolution by $\sharp$-filtered modules, then we conclude that $V$ must be $\sharp$-filtered through repeated applications of Corollary \ref{needtwo}.\\

Assume now that $V$ is presented in finite degree. It is clear that the first five statements imply that $H_i(V) = 0$ for some $i \geq 0$. Conversely, fix $i > 0$ such that $H_i(V) = 0$. If $i = 1$, then the previously proven equivalences imply that $V$ is $\sharp$-filtered. If $i > 1$, we may truncate a free resolution of $V$ at the $i$-th step, and write,
\begin{eqnarray}
0 \rightarrow X_{i-1} \rightarrow M_{i-2} \rightarrow \ldots \rightarrow M_0 \rightarrow V \rightarrow 0 \label{exact3}
\end{eqnarray}
where the $M_j$ modules are free. Note that the proof of Proposition \ref{homreg} as well as Theorem \ref{boundwidth} imply that we may assume all $M_j$, as well as $X_{i-1}$, are generated in finite degree. Using the fact that $H_1(X_{i-1}) = H_i(V) = 0$, we conclude that $X_{i-1}$ is $\sharp$-filtered.\\
\end{proof}

\subsection{The Proofs of Theorems \ref{shiftreg} and \ref{fistab}} \label{regdep}

The classification theorem will now allow us to prove the connection between regularity and depth.\\

\begin{proof}[Proof of Theorem \ref{shiftreg}]
Lemma \ref{shiftderv} tells us that there are isomorphisms for all $a$ and $b$,
\[
\So_b H_1^{D^a}(V) = H_1^{D^a}(\So_b V).
\]
The depth classification theorem implies that $\So_b V$ is sharp filtered if and only if $\So_b H_1^{D^a}(V) = 0$ for all $a\geq 1$. This is equivalent to saying that $\deg(H_1^{D^a}(V)) < b$.\\
\end{proof}

At this point we are ready to put all the pieces together. This theorem tells us that the question of when a module becomes sharp filtered is equivalent to bounding its derived regularity. Proposition \ref{homreg} then implies that bounds on the derived regularity leads to bounds on the homological degrees of the module. Moreover, we notice that this theorem can be used in both directions. Known bounds on how far a module must be shifted to become sharp filtered give bounds on the regularity of $V$. Conversely, known bounds on the derived regularity of $V$, such as those given in Theorem \ref{boundwidth}, give bounds on how far a module must be shifted to become $\sharp$-filtered. In particular, we have the following.\\

\begin{corollary}\label{filbound}
Let $V$ be an $\FI_G$-module which presented in finite degree, and is generated in degree $\leq d$. Then $\So_b V$ is sharp filtered for $b \geq hd_1(V) + \min\{hd_1(V),d\}$.\
\end{corollary}

This gives us all we need to bound the stable range when $G$ is a finite group.\\

\begin{proof}[The Proof of Theorem \ref{fistab}]
We recall that $\sharp$-filtered modules have a dimension polynomial for all $n$ (see the discussion immediately following Definition \ref{sfildef}). Corollary \ref{filbound}, and the techniques of Remark \ref{reldegdef}, immediately imply Theorem \ref{fistab}.\\
\end{proof} 

\begin{remark}
In \cite[Theorem 1.5]{L}, Li proves that if $G$ is a finite group, and if $b$ is such that $hd_i(\So_b V) \leq hd_0(\So_b V) + i$ for all $i$, then $hd_i(V) \leq hd_0(V) + b + i$ for all $i$. Theorems \ref{homacyclic} and \ref{polystab}, along with this theorem of Li, imply Theorem \ref{finitereg}. Indeed, by the results in this paper, Li's result implies
\[
hd_i(V) \leq hd_0(V) + \dreg(V) + 1 + i.
\]
This is related to the bound discussed in Theorem \ref{finitereg}, and is better in certain cases.
\end{remark}

\subsection{An Alternative Definition of Depth} \label{classical}

In this section we explore an alternative definition for depth, which is more classically rooted. For convenience of exposition, we assume throughout this section that $k$ is a field of characteristic 0, and that $G$ is a finite group. It is the belief of the author that all that follows can be done in the generality of the rest of the paper. The main result of this section will show that this alternative definition is equivalent to that given above.\\

If $R$ is a commutative ring with ideal $I \subseteq R$, it is classically known that for any $R$-module $M$,
\[
\de(I,M) = \min\{i \mid \Ext^i(R/I,M) \neq 0\}. \label{dc}
\]
See \cite[Proposition 18.4]{E} for the proof of this equality. This is the definition of depth which we will now emulate for our setting.\\

In \cite{GL}, the Gan and Li define the ring $k\FI_G = \bigoplus_{n \leq m} k[\Hom_{\FI_G}([n],[m])]$, whose multiplication is defined by
\[
(f,g)\cdot (f',g') = \begin{cases} (f,g) \circ (f',g') & \text{ if the codomain of $f'$ is equal to the domain of $f$}\\ 0 & \text{ otherwise.}\end{cases}
\]
Writing $e_n$ for the identity morphism in $\Hom_{\FI_G}([n],[n])$, we say that a $k\FI_G$-module $V$ is graded if $V = \bigoplus_n e_n \cdot V$. In the paper \cite{GL}, Gan and Li show that the category of graded $k\FI_G$-modules is equivalent to the category of $\FI_G$-modules.\\

One observes that the ring $k\FI_G$ has a natural two sided ideal generated by the upwards facing maps, $\mr = ((f,g):[n] \rightarrow [m] \mid n < m)$. The quotient $\G = k\FI_G/\mr$ is the direct sum $\bigoplus_n k[G_n]$ where all upwards facing maps act trivially.\\

\begin{definition}
We write $\G$ to denote the $\FI_G$-module which is defined on points by $\G_n = k[G_n]$, and whose transition maps are trivial. For any $\FI_G$-module $V$, The $k$-module $\Hom_{\FI_G\text{-Mod}}(\G,V)$ carries the structure of an $\FB_G$-module via
\[
\Hom_{\FI_G\text{-Mod}}(\G,V)_n = \Hom_{\FI_G\text{-Mod}}(k[G_n],V)
\]
We write $\Hom(\G,\dt):\FI_G\Mod \rightarrow \FB_G\Mod$ to denote this functor.\\
\end{definition}

\begin{proposition}\label{homprop}
The functor $\Hom(\G,\dt)$ enjoys the following properties:
\begin{enumerate}
\item $\Hom(\G,V) \neq 0$ if and only if $V$ has torsion;
\item $\Hom(\G,\dt)$ takes finitely generated $\FI_G$-modules to finitely generated $\FB_G$-modules;
\item $\Hom(\G,\dt)$ is right adjoint to the inclusion functor $\FB_G$-mod $\rightarrow \FI_G$-mod. In particular, $\Hom(\G,\dt)$ is left exact, and maps injective objects of $\FI_G$-mod to injective objects of $\FB_G$-mod.\\
\end{enumerate}
\end{proposition}

\begin{proof}
For the first statement, a map $\Hom_{\FI_G\text{-mod}}(k[G_n],V)$ is just a map of $k[G_n]$-modules $f:k[G_n] \rightarrow V_n$, whose image is in the kernel of every transition map out of $V_n$. Therefore if $\Hom(\G,V) \neq 0$, $V$ must have torsion. Conversely, if $v \in V_n$ is a torsion element, then there is some transition map $(f,g)_\as:V_n \rightarrow V_m$ such that $(f,g)_\as(v) = 0$. We may write the map $(f,g)$ as 
\[
(f,g) = (f_1,g_1) \circ (f_2,g_2)
\]
where $(f_1,g_1):[m-1] \rightarrow [m]$ and $(f_2,g_2):[n] \rightarrow [m-1]$. If $(f_2,g_2)_\as(v) = 0$, then we may repeat this decomposition until it is not. Assuming that this is not the case, then $(f_2,g_2)_\as(v) \in V_{m-1}$ is in the kernel of all transition facing maps. In particular, $\Hom(\G,V)_{m-1} \neq 0$.\\

Any element of $\Hom(\G,V)_n$ will naturally correspond to a torsion element of $V_n$ by the previous discussion. In particular, we may consider $\Hom(\G,V)$ as a submodule of $V$. The Noetherian Property implies that $\Hom(\G,V)$ is finitely generated.\\

Let $W$ be a finitely generated $\FB_G$-module, and let $V$ be a finitely generated $\FI_G$-module. If we consider $W$ as being an $\FI_G$-module with trivial transition maps, then a morphism $W \rightarrow V$ must send elements of $W$ to elements of $V$ which are in the kernel of all transition maps. By the previous discussion, these correspond precisely to the elements of $\Hom(\G,V)$. Conversely, if we have a map $\phi:W \rightarrow \Hom(\G,V)$, then we set $\widetilde{\phi}:W \rightarrow V$ by $\widetilde{\phi}_n(w) = \phi_n(w)(id)$.\\

The final two statements are general facts from homological algebra. Any functor which is right adjoint has the first property, and if its left adjoint is exact then it has the second property.\\
\end{proof}

We have already discussed the fact that the category $\FI_G$-mod might not have sufficiently many injective objects if $k$ is a field of characteristic $p > 0$. When $k$ is a field of characteristic 0, and $G$ is a finite group, however, this is not an issue. In fact, Sam and Snowden \cite[Theorem 4.3.1]{SS3}, as well as Gan and Li \cite[Theorem 1.7]{GL}, have shown that every object has finite injective dimension in this case. Moreover, the cited papers prove that all projective objects are also injective, and these are precisely the torsion free injective objects. In other words, one has the following chain of equalities
\[
\{\text{$\sharp$-Filtered Modules}\} = \{\text{Projective Modules}\} =\{\text{Torsion Free Injective Modules}\}
\]
\text{}\\

\begin{definition}
We will denote the right derived functors of $\Hom(\G,\dt)$ by $\Ext^i(\G,\dt)$. Given a finitely generated $\FI_G$-module $V$, we define its \textbf{classical depth} to be the quantity $\dec(V) := \min\{i \mid \Ext^i(\G,\dt) \neq 0\}$. \\
\end{definition}

As was the case with the other notion of depth, the first property one needs to prove to justify using the name depth is the increment property.\\

\begin{lemma}\label{adepth2}
If there is an exact sequence of finitely generated $\FI_G$-modules
\[
0 \rightarrow V' \rightarrow X \rightarrow V \rightarrow 0
\]
such that $X$ is $\sharp$-filtered, then $\dec(V') = \dec(V) + 1$.\\
\end{lemma}

\begin{proof}
Applying the functor $\Hom(\G,\dt)$ to the above sequence, and using the fact that $X$ is injective and torsion free, we find that $\Ext^{i+1}(\G,V') = \Ext^i(\G,V)$ for all $i \geq 0$. This immediately implies the desired result.\\
\end{proof}

The first step in showing that the two notions of depth coincide is to show that they are the same at the extremes. Proposition \ref{homprop} implies that depth zero and classical depth zero are equivalent. We next prove that objects of infinite classical depth are precisely the $\sharp$-filtered modules.\\

\begin{proposition}
For any finitely generated $\FI_G$-module $V$, $\dec(V) = \infty$ if and only if $V$ is $\sharp$-filtered.\\
\end{proposition}

\begin{proof}
Because $\dec(V)$ is defined in terms of the vanishing of $\Ext$ functors, it follows that $\dec(V) = \infty$ if $V$ is injective and torsion free. Conversely assume that $V$ has infinite classical depth. This implies, in particular, that $V$ is torsion free, and therefore embeds into all its shifts. Using Theorem \ref{polystab}, we may find some $b >> 0$ such that $\So_b$ is $\sharp$-filtered. Lemma \ref{inddeg} implies the cokernel of $V \hookrightarrow \So_b V$ is generated in strictly smaller degree, while Lemma \ref{adepth2} implies that it has infinite classical depth. Proceeding inductively, we will eventually be left with a module $V'$ such that $\So_{b'}V'$ is sharp filtered for some $b' >> 0$ and $\So_{b'}V'/ V'$ is $\sharp$-filtered. Corollary \ref{needtwo} implies that $V$ was $\sharp$-filtered to begin with.\\
\end{proof}

This is all we need to prove the equivalence.\\

\begin{theorem}
Let $V$ be a finitely generated $\FI_G$-module. Then $\dec(V) = \de(V)$.\\
\end{theorem}

\begin{proof}
If $\dec(V) = 0$ or $\dec(V) = \infty$, then we have already seen $\de(V)$ agrees with $\dec(V)$. Assume that $\dec(V) = \delta > 0$, and assume for the sake of contradiction that $\de(V) = \delta' < \delta$. Because $\delta$ is positive, it must be the case that $\delta'$ is positive as well. We may therefore find some $b >> 0$ such that $\So_b V$ is $\sharp$-filtered. This gives us an embedding
\[
0 \rightarrow V \rightarrow \So_b V \rightarrow Q \rightarrow 0
\]
for some $Q$ with $\dec(Q) = \delta - 1$, and $\de(Q) = \delta' - 1$. Proceeding inductively, we would eventually we left with a module with $\de(V') = 0$ and $\dec(V) > 0$. This is a contradiction.\\
\end{proof}

\begin{remark}
In \cite{SS3}, Sam and Snowden provide a definition of the depth of an $\FI$-module in characteristic 0. This definition is formulated in terms of a kind of Auslander-Buchsbaum formula. It can be shown that this definition of depth is also equivalent to the definitions given in this paper.\\
\end{remark}

\end{document}